\documentclass[11pt,a4paper]{amsart}
\usepackage{hyperref}
\usepackage[textformat=period]{caption}
\usepackage[english,british]{babel}
\usepackage[utf8]{inputenc}

\usepackage{amsmath,amsthm,amssymb,amsfonts}
\usepackage{dsfont}
\usepackage{textcomp}
\usepackage{graphicx}
\usepackage{extarrows}
\usepackage{epigraph}
\usepackage{graphicx}
\usepackage{subfigure}
\usepackage{sidecap}
\usepackage{indentfirst}
\usepackage{tikz}
\usepackage{pgfplots}
\usetikzlibrary{arrows,matrix,patterns,decorations.markings}
\usepackage{caption}
\usepackage{float}
\usepackage{color}

\newcommand{\Z}{\mathds{Z}}
\newcommand{\Q}{\mathds{Q}}
\newcommand{\OO}{\mathds{O}}
\newcommand{\dd}{\mathrm{d}}

\newcommand{\Hl}{H_{\text{Lee}}}

\newcommand{\ssum}{\displaystyle\sum}

\renewcommand{\geq}{\geqslant}
\renewcommand{\leq}{\leqslant} 

\newtheorem{teo}{Theorem}[section]
\newtheorem*{teo*}{Theorem}
\newtheorem{lemma}[teo]{Lemma}
\newtheorem{prop}[teo]{Proposition}
\newtheorem*{prop*}{Proposition}
\newtheorem{defin}{Definition}

\newtheorem{cor}[teo]{Corollary}

\newcommand{\osplit}[2]{\begin{scope}[shift={(#1+2,#2+2)}]
\draw (-2,-2) .. controls +(1,1) and +(-1,1) ..  (2,-2);
\draw (-2,+2) .. controls +(1,-1) and +(-1,-1) ..  (2,+2);
\end{scope}}

\newcommand{\isplit}[2]{\begin{scope}[shift={(#1+2,#2+2)}]
\draw (-2,-2) .. controls +(1,1) and +(1,-1) ..  (-2,+2);
\draw  (2,-2).. controls +(-1,1) and +(-1,-1) ..  (2,+2);
\end{scope}}

\newcommand{\ocross}[2]{\begin{scope}[shift={(#1+2,#2+2)}]
\draw (+2,-2) -- (-2,+2);
\pgfsetlinewidth{8*\pgflinewidth}
\draw[white] (-2,-2) -- (+2,+2);
\pgfsetlinewidth{.125*\pgflinewidth}
\draw (-2,-2) -- (+2,+2);
\end{scope}}

\newcommand{\ocrosstext}{\raisebox{-0.1cm}{\begin{tikzpicture}[scale=.07]
\ocross{0}{0}
\draw[dashed] (2,2) circle (2.8cm);
\end{tikzpicture}}\:}

\newcommand{\smootho}{\raisebox{-0.1cm}{\begin{tikzpicture}[scale=.07]
\osplit{0}{0}
\draw[dashed] (2,2) circle (2.8cm);
\end{tikzpicture}}\:}
\newcommand{\smoothv}{\raisebox{-0.1cm}{\begin{tikzpicture}[scale=.07]
\isplit{0}{0}
\draw[dashed] (2,2) circle (2.8cm);
\end{tikzpicture}}\:}

\begin{document}
\title{On the slice genus and some concordance invariants of links}
\author{Alberto Cavallo}

\begin{abstract}
 We introduce a new class of links for which we give a lower bound for the slice genus $g_*$, using the generalized 
 Rasmussen invariant. We show that this bound, in some cases, allows one to compute $g_*$ exactly; in particular, we 
 compute $g_*$ for torus links. 
 We also study another link invariant: the strong slice genus $g_*^*$. Studying the behaviour of a specific type of
 cobordisms in Lee homology, a lower bound for $g_*^*$ is also given. 
\end{abstract}

\date{\today}
\address{Department of Mathematics and its Applications, CEU, Budapest, 1051, Hungary}
\email{Cavallo\_Alberto@ceu-budapest.edu}

\maketitle

\section{Introduction}
\label{section:intro}
In the last ten years, after the publication of the paper ``Khovanov homology and the slice genus'' \cite{Rasmussen} by 
Jacob Rasmussen, based on Eun Soo Lee's work (\cite{Lee}), more concordance invariants of links have been studied: 
in \cite{Beliakova}, Beliakova and Wehrli introduced a natural generalization to oriented links of the Rasmussen invariant $s(L)$ and, more recently, 
John Pardon introduced a new invariant called $d_L$. 

After a review on link cobordisms and Lee homology (Section \ref{section:cobordisms}),
in Section \ref{section:three} we use $d_L$ to define a new class of links, that we call
\emph{pseudo-thin links}: a very large class that contains both knots and Khovanov H-thin links. We prove 
more properties of $s(L)$ than the ones given in \cite{Beliakova} for this class of links.
Furthermore, we show that for pseudo-thin links, the inequalities due to Andrew Lobb (\cite{Lobb}) can be improved in the following way:
\begin{equation}
\label{inf_s}
s(D,{\bf o})\geq 2-2r+V(D,{\bf o})
\end{equation} 
where $r$ is the number of split component of the link represented by $(D,{\bf o})$
and $V(D,{\bf o})$ is an integer obtained from the diagram as in \cite{Lobb}.

In Section \ref{section:four}, we compute, for some families of links, the slice genus $g_*(L)$, 
namely the minimum genus of a compact connected orientable surface 
$S$ embedded in the 4-ball $D^4$ with $\partial S=L$.  
\begin{figure}[H]
  \centering
  \includegraphics[width=0.38\textwidth]{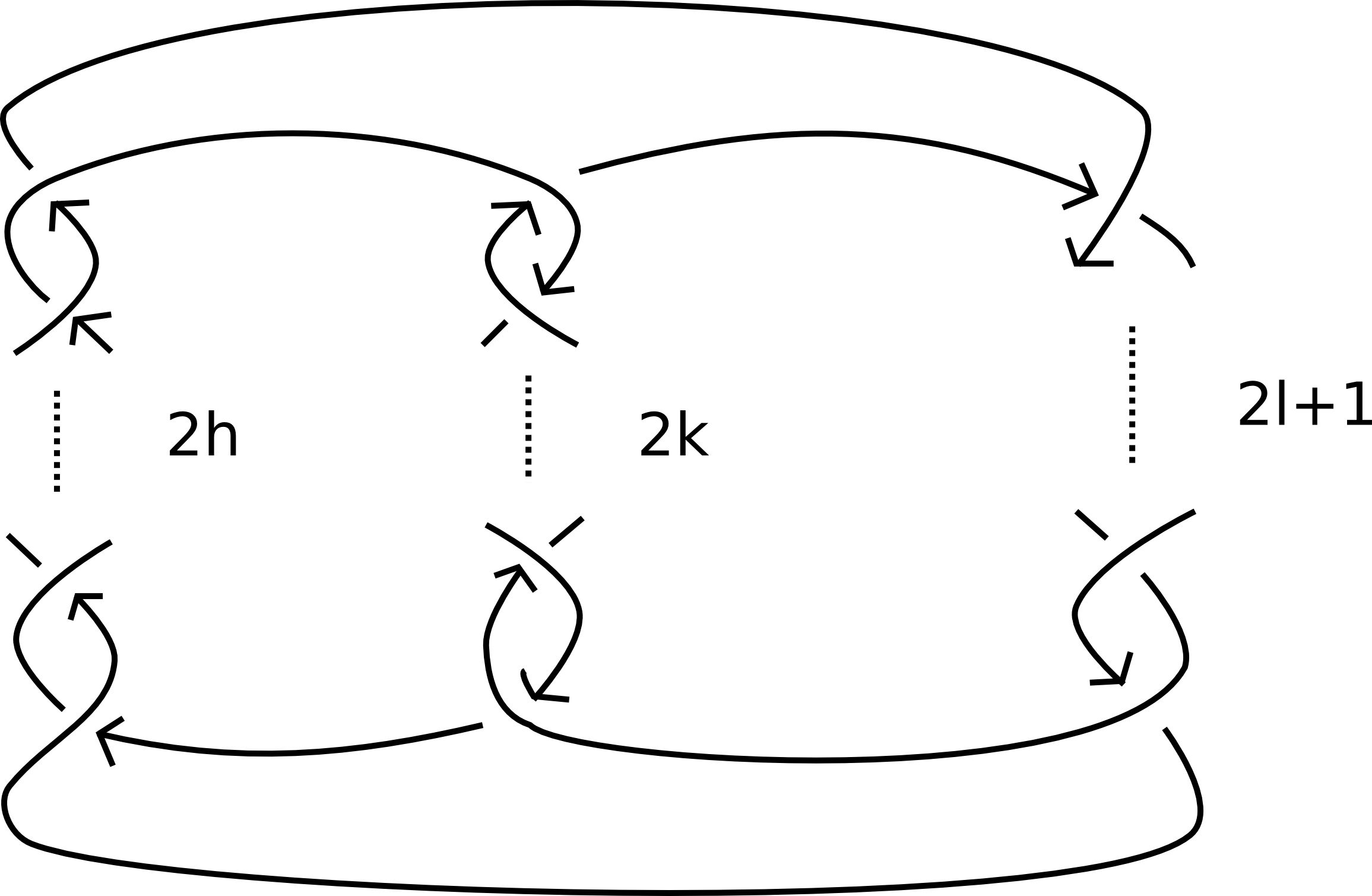}   
  \caption{A $P_{2h,-2k,2l+1}$ pretzel link: the number of crossings of each strand is indicated}
  \label{Pretzel}
\end{figure}
With our first result we verify a generalization to links of the classical Milnor conjecture on torus knots.
\begin{prop}
 \label{prop:Milnor}
 The slice genus of a $n=gcd(p,q)$ component torus link $T_{p,q}$ is given by the following equation:
 \begin{equation}
  \label{Milnor}
  \dfrac{(p-1)(q-1)+1-n}{2}
 \end{equation}
\end{prop}
Then, we consider the 3-strand pretzel links given in Figure \ref{Pretzel}.
Note that the second parameter is negative because the second strand has negative twists, while we can always suppose $l\geq 0$ because
$P_{-2h,-2k,-2l-1}$ is the mirror image of $P_{2h,2k,2l+1}$ for every $(h,k,l)$.

Under the hypothesis that $h<0$ and $k<0$ we prove the following proposition, using Equation \eqref{inf_s} and 
Proposition~\ref{prop:Milnor}. Later, we will show that more can be said on pretzel slice genus.
\begin{prop}   
 \label{prop:boh}
       If $l+h\geq 0$ then $$g_*(P_{2h,2k,2l+1})=l+h$$ while if $l+h=-1$ then $g_*(P_{2h,2k,2l+1})=0$.
       Changing the relative orientation, if $l+k\geq 0$, we have 
       $$g_*(P_{2h,2k,2l+1})=l+k$$ and $l+k=-1$ implies $g_*(P_{2h,2k,2l+1})=0$.     
\end{prop}
Finally, let $Tw_n$ with $n\geq 0$ be the link of Figure~\ref{Tw_n}.
\begin{figure}[H]
  \centering
  \includegraphics[width=0.73\textwidth]{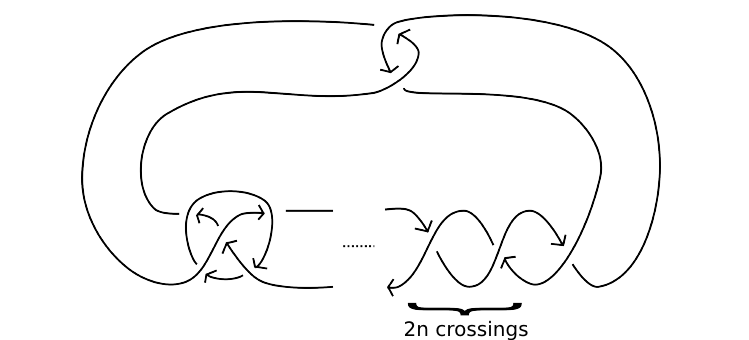}  
  \caption{The link $Tw_n$}
  \label{Tw_n}
\end{figure}
$Tw_{n}$ is a two component non split link with a twist knot and an unknot as components. Using Equation \eqref{inf_s} we have
$s(Tw_n)=3$ for all $n$ and, as we will see, this implies $g_*(Tw_n)\geq 1$. Indeed for $n=1$ we can compute the precise 
value of the slice genus, that 
is $g_*(Tw_1)=1$.
All the links we mentioned before are pseudo-thin, but in general they are not H-thin.

We say that a cobordism between two $n$ components links is strong if there exist $n$ disjoint knot 
cobordisms between a component of the first link and one of the second link. We study these cobordisms in 
Section \ref{section:strong} where we prove the following theorem.
\begin{teo}
 \label{teo:not_concordant}
 Let $L$ be a pseudo-thin link with $h$ split components and let $\Sigma$ be a strong cobordism between $L$ and
 $M$, where $M$ is a link with $k$ split components.
 Then $g(\Sigma)\geq\lceil\frac{k-h}{2}\rceil$, where $g(\Sigma)$ is the genus of the surface $\Sigma$. In particular, every
 non split pseudo-thin link is not strongly concordant, namely if there exists a genus zero strong cobordism, to a split link.
\end{teo}
This is a generalization of Pardon's main result in \cite{Pardon};  using the same ideas, we also provide a lower bound for 
the genus of a 
strong cobordism between a pseudo-thin link and  
the $n$ component unlink:

\newpage

\begin{equation}
 \label{boh}
 g_*^*(L)\geq\dfrac{|s(L)|+n-1}{2}\:.
\end{equation}
This bound leads us to study an invariant of links similar to
the slice genus, but considering only strong cobordisms: the strong slice genus $g_*^*$.

At last, in Section \ref{section:examples}, we compute this invariant for an infinite family of pretzel links:
$$g_*^*(P_{2h,2h,2l+1})=l+h\:\:\:\:\:\text{ if }\:\:\:\:\:h,l\geq 0\:.$$

\subsection*{Acknowledgements}
This paper is an expanded excerpt from a Master's thesis prepared at the University of Pisa under
the supervision of Paolo Lisca, to whom many thanks are due for his advice and support.

\section{Cobordisms and filtered Lee homology}
\label{section:cobordisms}
\subsection{Link cobordism}
First we give the definitions of weak and strong cobordism.
\begin{defin}
 Let $L,L'$ be links such that  
 $$L:U=\underbrace{S^1\sqcup...\sqcup S^1}_{n\:\text{ times}}\rightarrow S^3\:\:\:\: 
 \text{ and }\:\:\:\: L':V=\underbrace{S^1\sqcup...\sqcup S^1}_{m\:\text{ times}}\rightarrow S^3\:.$$
 
 A weak cobordism of genus $g$ between $L$ and $L'$ is a smooth embedding 
 $$f:S_{g,n+m}\rightarrow S^3\times I$$ where $S_{g,n+m}$ is a compact orientable surface of genus $g$ in which
 $\partial S_{g,n+m}=-U\sqcup V$, every connected component of $S_{g,n+m}$ has boundary in $L$ and $L'$, and 
 $f(U)=L(U)\times\{0\}$ and $f(V)=L'(V)\times\{1\}$.
 
 $L$ and $L'$ are said to be weakly concordant if there exists a weak cobordism of genus $0$ between them.
 
 A cobordism is strong if the links also have the same number of components and if they satisfy the condition 
 $$S_{g,n+n}=S_{g_1,1+1}\sqcup...\sqcup S_{g_n,1+1}$$ with $g_1+...+g_n=g$.
 
 $L$ and $L'$ are strongly concordant if there exists a strong cobordism of genus $0$ between them.
\end{defin}
It is clear that the two definitions of cobordism are the same for knots. This means that there is no 
ambiguity in saying that a slice knot is a knot which is concordant to the unknot.

Since $H_1(D^4,S^3,\Z)\cong\{0\}$ it is easy to see that there is always a properly embedded orientable surface  
in $D^4$ with a given link $L$ as boundary: the minimum genus for this kind of surface is called the slice genus of $L$ 
and it is denoted with $g_*(L)$. Then we say that a link is slice if $g_*(L)=0$, that is if it is weakly concordant to the 
unknot. We have 
$g_*(L)\leq g(L)$ where $g(L)$ is the classic 3-genus of $L$. 

$\newline$ The signature of a link $L$ is the integer $\sigma(L)=\text{sgn }(A+A^T)$, where $A$ is a Seifert matrix
for $L$. Then, two strongly concordant links have the same signature; in particular a slice knot $K$ has $\sigma(K)=0$.
The signature of a link is a strong concordance invariant. In this paper we will see some other examples of such 
invariants.

\subsection{Link homology theories}
It is well known that we can associate a bigraded $\Q$-cochain complex $(C(D),\dd)$ to an oriented diagram $D$ of a link. The 
process has first 
been introduced by Mikhail Khovanov in \cite{Khovanov}, where he defined his own homology, and it is not described again 
in this paper. \newpage The differential $\dd$ is obtained from the following maps:
$$\begin{aligned}  
    &\:V\otimes V\xlongrightarrow{m}V \\
    &v_+\otimes v_+\rightarrow v_+ \\
    &v_+\otimes v_-\rightarrow v_- \\
    &v_-\otimes v_+\rightarrow v_- \\
    &v_-\otimes v_-\rightarrow 0
   \end{aligned}
  \hspace{1.5 cm}
   \begin{aligned}
    &\:\:\:\:\:\:\:\:V\xlongrightarrow{\Delta}V\otimes V \\
    &v_+\rightarrow v_+\otimes v_-+v_-\otimes v_+\\
    &v_-\rightarrow v_-\otimes v_-
   \end{aligned}
  $$
where $V$ is the two-dimensional $\Q$-space associated to a circle in each resolution of $D$. The maps $m$ and $\Delta$ are the
multiplication and the comultiplication of a Fr\"obenius algebra. This is enough to claim that $\dd$ is actually a differential; further details can 
be found 
in \cite{Khovanov} and \cite{Lee}. 

Rational Khovanov homology is the homology of $(C(D),\dd)$:
$$H^{i,j}(D)=\dfrac{\text{Ker }\dd^{i,j}}{\text{Im }\dd^{i-1,j}}$$
and it is invariant under Reidemeister moves.

We ask how we can change the maps $m$ and $\Delta$ so that they extend to a 
Fr\"obenius algebra and that, at the same time, the homology of the new
complex remains a link invariant; in this case $H_*(C(D),\dd)$ is called a homology link theory. The answer is that we must have
$$m_{(h,t)}:\left\{
   \begin{aligned}
    &v_+\otimes v_+\rightarrow v_+ \\
    &v_+\otimes v_-\rightarrow v_- \\
    &v_-\otimes v_+\rightarrow v_- \\
    &v_-\otimes v_-\rightarrow hv_-+tv_+
   \end{aligned}
 \right.\:\:\:\:\Delta_{(h,t)}:\left\{
  \begin{aligned}
    &v_+\rightarrow v_+\otimes v_-+v_-\otimes v_+ \\
    &\:\:\:\:\:\:\:\:\:\:\:\:\:-hv_+\otimes v_+ \\
    &v_-\rightarrow v_-\otimes v_-+tv_+\otimes v_+
   \end{aligned}
  \right.
\:(h,t)\in\Q\times\Q\:.$$ According to Turner \cite{Turner} the following theorem holds.
\begin{teo}
 There are only two, up to isomorphism, homology link theories: Khovanov homology, with pairs $(h,t)$ such that $h^2+4t=0$ and
 Lee homology, with $(h,t)$ such that $h^2+4t\neq 0$.
\end{teo}
We define Lee homology considering $\dd$ obtained by the pair $(0,\frac{1}{4})$, this simply for numeric convenience. 
The differential $\dd_{\text{Lee}}$ is graded, thus we have $$\Hl^i(L)=\dfrac{\text{Ker }\dd_{\text{Lee}}^i}{\text{Im }\dd_{\text{Lee}}^{i-1}}\:.$$
From \cite{Lee}, we know that $\text{dim }\Hl(L)$=$2^n$, where $n$ is the number of components of $L$. Moreover, taking 
$2\cdot\text{lk}({\bf o'},{\bf o})=n_-(D,{\bf o'})-n_-(D,{\bf o})$ where ${\bf o}$ and ${\bf o'}$ are two orientations of a 
diagram
$D$ of $L$, we also have (\cite{Lee}) the following proposition.
\begin{prop}
 $$\text{dim }H_{\text{Lee}}^i(L)=\bigl|\left\{{\bf o'}\in\OO(D)\:|\:2\cdot\text{lk}({\bf o'},{\bf o})=i\right\}\bigr|$$
 and
 $$H_{\text{Lee}}^i(L)=\text{Span}\left\{[v({\bf o'})]\in\Hl(L)\:|\:2\cdot\text{lk}({\bf o'},{\bf o})=i\right\}$$
 where $(D,{\bf o})$ is an oriented diagram of $L$ and $\OO(D)$ is the set of all $2^n$ orientations of $D$.
 $\newline$ For every ${\bf o}\in\OO(D)$ $[v({\bf o})]$ is the canonical generator of $\Hl(L)$ associated to ${\bf o}$ as 
 described in \cite{Rasmussen}.
\end{prop}
Lee homology has been defined by Eun Soo Lee in order to prove an important conjecture on H-thin links, namely links whose 
Khovanov 
invariant $H^{i,j}(L)$ is supported in two diagonal lines $j=2i+s\pm 1$ for some integer $s$; see Lee's paper for more 
informations. Our goal is to show some other important applications of this theory.

\subsection{A filtration on \texorpdfstring{$\Hl(L)$}{Lg}}
Rasmussen has equipped Lee homology with the structure of a decreasing filtration. Let $(C(D),\dd_{\text{Lee}})$ be the
Lee complex for a link diagram $D$; for all $s\in\Z$ we take 
$$\mathcal{F}^sC^i(D)=\text{Span }\{v\in C^{i,j}(D)\:|\:j\geq s\}\:.$$
It is easy to see that $\mathcal{F}^sC^i(D)\supset\mathcal{F}^{s+1}C^i(D)$ and, from \cite{Rasmussen}, we have 
$$\dd_{\text{Lee}}^i(\mathcal{F}^sC^i(D))\subset\mathcal{F}^sC^{i+1}(D)$$ 
so the filtration descends to homology:
$$\mathcal{F}^sH^i_{\text{Lee}}(D)=\dfrac{\mathcal{F}^s\text{Ker }\dd_{\text{Lee}}^i}{\mathcal{F}^s\text{Im }\dd_\text{Lee}^{i-1}}\:.$$
For $[v]\in H^i_{\text{Lee}}(D)$ we say that $\text{sdeg}[v]=s$ if
$[v]\in\mathcal{F}^sH^i_{\text{Lee}}(D)\setminus\mathcal{F}^{s+1}H^i_{\text{Lee}}(D)$ that is
$$\mathcal{F}^sH^i_{\text{Lee}}(D)=\text{Span}\{[v]\in H^i_{\text{Lee}}(D)\:|\:\text{sdeg}[v]\geq s\}\:.$$
\begin{defin}
 Let $(\mathcal{C},\dd,\mathcal{F})$ and $(\mathcal{C}',\dd',\mathcal{F})$ be cochain complexes with a filtration, then a map of complexes 
 $f:\mathcal{C}\rightarrow\mathcal{C}'$ is 
 a filtered map of degree $n$ if it satisfies
 $f(\mathcal{F}^s\mathcal{C}^i)\subset\mathcal{F}^{s+n}(\mathcal{C}')^i$ for every $i$.
 We say that $f$ respects the filtration if it is filtered of degree 0.
\end{defin}
It is obvious that, if $f:\mathcal{C}\rightarrow\mathcal{C}'$ is a filtered map of degree $n$, 
then $f_*(\mathcal{F}^sH_i(\mathcal{C}))\subset\mathcal{F}^{s+n}H_i(\mathcal{C}')$ for every $i$, so that the induced map in homology has 
the same degree.
$\newline$ Rasmussen (\cite{Rasmussen}) has also proved that filtered Lee homology is a link invariant.
\begin{teo}
 If $D$ and $D'$ are two diagrams of equivalent links then 
 $$\mathcal{F}^sH^i_{\text{Lee}}(D)\cong\mathcal{F}^sH^i_{\text{Lee}}(D')\:\:\:\:\:\text{for every }i,s\in\Z$$
\end{teo}
in fact, if $D$ and $D'$ differ by a Reidemeister move, all the isomorphisms in Lee homology set in \cite{Lee} and their inverses 
respect the filtration.

From \cite{Rasmussen} we have the following two propositions. 
\begin{prop}
 \label{prop:filtered_homology1}
 Let $L$ and $L'$ be links with diagrams $D$ and $D'$ and let $\Sigma$ be a weak cobordism between them; then there is 
 $$F_{\Sigma}:(C(D),\dd_{\text{Lee}},\mathcal{F})\longrightarrow(C(D'),\dd_{\text{Lee}},\mathcal{F})$$ a filtered map of
 degree $\chi(\Sigma)$ which induces
 $$(F_{\Sigma}^i)^*:\Hl^i(L)\rightarrow\Hl^i(L')$$ filtered of degree $\chi(\Sigma)$. 
\end{prop}
\begin{prop}
 \label{prop:filtered_homology}
 If $\Sigma$ is a strong cobordism between $L$ and $L'$ then $$(F_{\Sigma}^i)^*:\Hl^i(L)\rightarrow\Hl^i(L')$$ is a filtered 
 isomorphism of degree $\chi(\Sigma)$.
\end{prop}
Corollary \ref{cor:boh} follows immediately from Proposition \ref{prop:filtered_homology1} and \ref{prop:filtered_homology}.
\begin{cor}
 \label{cor:boh}
 If $L$ and $L'$ are strongly concordant then $\chi(\Sigma)=0$. In particular, $(F_{\Sigma}^i)^*$ and its inverse respect the 
 filtration and 
 $$\mathcal{F}^s\Hl^i(L)\cong\mathcal{F}^s\Hl^i(L')\:\:\:\:\:\forall i,s\in\Z\:.$$
 Furthermore, $(F_{\Sigma}^i)^*$ takes canonical generators of $\Hl(L)$ into canonical generators of $\Hl(L')$.  
\end{cor}
This means that filtered Lee homology is a strong concordance invariant.

\section{\texorpdfstring{$s$}{Lg} invariant lower bound for pseudo-thin links}
\label{section:three}
\subsection{Invariants from Lee homology}
We expect filtered Lee homology to be a powerful invariant, but it is very difficult to fully understand. 
What we do is to study two other invariants that are numeric reductions of $\mathcal{F}\Hl$. Thus,
they preserve the property of being a strong concordance invariant, combined with the advantage of being easier to 
compute.
\subsubsection{The Pardon invariant}
Let $D$ be an oriented diagram of a link $L$, then the function 
$$d_D:\Z\times\Z\rightarrow\Z_{\geq0}$$ 
$$d_D(i,s)=\text{dim }\dfrac{\mathcal{F}^sH^i_{\text{Lee}}(D)}{\mathcal{F}^{s+1}H^i_{\text{Lee}}(D)}$$ 
is the link concordance invariant defined by John Pardon in \cite{Pardon}. He proved the following properties.
\begin{prop}
 Let $L_1,L_2$ be links and $L$ be another link with orientation ${\bf o}$. Then we have
 \begin{equation}
 \label{d_support}
 \ssum_{s\equiv n+k\:(4)}d_L(i,s)=\left\{\begin{aligned}
                                                      &0\hspace{2.5cm}\text{ if }k\equiv 1\:\:\:\:\:\text{mod }2\\
                                                      &\dfrac{1}{2}\text{dim }\Hl^i(L)\:\:\:\:\text{ if }k\equiv 0\:\:\:\:\:\text{mod }2\\
                                                     \end{aligned}\right.
 \end{equation}
 \begin{equation}
  \label{d_mirror}
  d_{L^*}(i,s)=d_L(-i,-s)\hspace{1.8 cm}\text{ for all }i,s\in\Z
 \end{equation}
 \begin{equation}
  \label{disjoint_union}
  d_{D_1\sqcup D_2}(i,s)=(d_{D_1}*d_{D_2})(i,s)\:\:\:\:\:\text{for all }i,s\in\Z
 \end{equation}
 \begin{equation}
  d_{(D,{\bf o'})}(i,s)=d_{(D,{\bf o})}\left(i+2\cdot\text{lk}({\bf o'},{\bf o}),s+6\cdot\text{lk}({\bf o'},{\bf o})\right)
 \end{equation}
 where $*$ is convolution, ${\bf o'}$ is another orientation of $L$, while $L_1\sqcup L_2$ and $L^*$ stand 
 for disjoint union and mirror of links respectively.
\end{prop}
It is easy to see that the unknot has $d(i,s)\neq 0$ if and only if $(i,s)=(0,\pm 1)$. Using \eqref{disjoint_union} we obtain
the values of $d$ for the $n$ component unlink $\bigsqcup^n\bigcirc$:
$$d_{\bigsqcup^n\bigcirc}(i,s)=\left\{\begin{aligned}
                                                   &\binom{n}{k}\:\text{ if }i=0,s=n-2k \\
                                                   &0\:\:\:\:\:\:\:\:\text{ otherwise}\\
                                                  \end{aligned}\right.
\:\:\:\:\:\text{with } k=0,...,n\:.$$
\subsubsection{The generalized Rasmussen invariant}
This is the invariant introduced by Anna Beliakova and Stephan Wehrli in \cite{Beliakova}.

Let $(D,{\bf o})$ be a diagram of a link $L$, the generalized Rasmussen invariant is the integer 
$$s(L)=\dfrac{1}{2}\bigl(\text{sdeg}[v({\bf o})+v({\bf -o})]+\text{sdeg}[v({\bf o})-v({\bf -o})]\bigr)\:.$$
From \cite{Lee} and \cite{Rasmussen} we know that H-thin links whose $H^{i,j}$ is non zero in $j=2i+c\pm 1$ 
have $d$ invariant supported in lines $s=2i+c\pm 1$. This implies that $s(L)$ is necessarily equal to $c$ and this leads us to the following 
corollary.
\begin{cor}
 If $L$ is a non split alternating link then $s(L)=-\sigma(L)$.
\end{cor}
We can also compute $s(L)$ when $L$ has a positive diagram. Recall that $\smoothv$ and $\smootho$ are said to be the
0-resolution and 1-resolution of $\ocrosstext$; the coherent resolution of an oriented diagram is the diagram obtained by performing 
the only orientation preserving resolution at every crossing.
\begin{prop}
 \label{prop:positive}
 Let $(D,{\bf o})$ be a positive diagram, then $$s(D,{\bf o})=-k(D)+c(D)+1$$ where $k(D)$ is the number of circles of the
 coherent resolution of $(D,{\bf o})$ and $c(D)$ is the number of crossing of the diagram.
\end{prop}
\begin{proof}
 $D$ is positive thus $v({\bf o})\pm v({\bf -o})\in C^0(D)$. Further we have $\text{Im }\dd_{\text{Lee}}^{-1}=\{0\}$ whence 
 $\text{sdeg}[v({\bf o})\pm v({\bf -o})]=\{j,j+2\}$ with $v({\bf o})\pm v({\bf -o})\in C^{i,j}(D)\cup C^{i,j+2}(D)$.
 $\newline$ The shift in the complex $C(D)$ tells us that $j=-k(D)+c(D)$ and the statement follows.
\end{proof}
As an example, for the $n$ component unlink we have $s(\bigsqcup^n\bigcirc)=1-n$.

In \cite{Beliakova} the following properties of $s$ are proved.
\begin{prop}
 Let $L,L_1,L_2$ be links and $n$ be the number of components of $L$. Then the following inequalities hold
 \begin{equation}
 \label{cobordismo_s}
         \left|s(L_1)-s(L_2)\right|\leq-\chi(\Sigma)
 \end{equation} 
 \begin{equation}
 \label{s_disjoint_union}
         s(L_1\sqcup L_2)=s(L_1)+s(L_2)-1
 \end{equation}
 \begin{equation}
 \label{connected_sum}
         s(L_1)+s(L_2)-2\leq s(L_1\sharp L_2)\leq s(L_1)+s(L_2)
 \end{equation}
 \begin{equation}
 \label{mirror}
         2-2n\leq s(L)+s(L^*)\leq 2
 \end{equation}
 where $\Sigma$ is a weak cobordism between $L_1$ and $L_2$, while $L_1\sharp L_2$ stay  
 for connected sum of links.
\end{prop}

\subsection{Pseudo-thin links}
Now we are going to introduce a new class of links for which we can use $d_L$ and its properties to improve Inequalities \eqref{connected_sum} and \eqref{mirror}. %\newpage
\begin{defin}
 Let $L$ be a non split link. We say that $L$ is pseudo-thin if for all $[v]\in\Hl^0(L)$ we have $\text{sdeg}[v]=c(L)\pm 1$ 
 with $c(L)\in\Z$. This means that $d_L(0,\cdot)$ is supported in two points. 
 A link is pseudo-thin if each of its split components is pseudo-thin.
\end{defin} 
We see that, given an oriented diagram $(D,{\bf o})$ of a pseudo-thin link, we have $s(D,{\bf o'})=c$ for all 
${\bf o'}\in\OO(D)$ such that $\text{lk}({\bf o'},{\bf o})=0$. Moreover, every knot or H-thin link is pseudo-thin.
We have already seen in Section \ref{section:intro} examples of non split pseudo-thin links which are not H-thin.
\subsubsection{Mirror of pseudo-thin links} 
The first proposition we prove is a generalization of Rasmussen's result on knots.
\begin{prop}
 \label{prop:mirror_pseudo}
 If $L$ is a pseudo-thin link then $$s(L^*)=2-2r-s(L)$$ where $r$ is the number of split component of $L$.  
\end{prop}
\begin{proof}
 We proceed by induction on $r$. 
 \begin{itemize}
  \item $r=1$
  
        This is the non split case.
        
        $d_{L}(0,s)$ is non zero only for $s=s(L)\pm 1$ since $L$ is pseudo-thin. From \eqref{d_mirror} we see that 
        $d_{L^*}(0,s)=d_L(0,-s)$ thus $d_{L^*}(0,s)$ is non zero only for $s=-s(L)\pm 1$ and, from this, we have
        $$s(L^*)=-s(L)\:.$$
  \item $r\geq 2$
  
        $L=L_1\sqcup L_2$. We can suppose $L_1$ is non split and $L_2$ has $r-1$ split components. By Equation \eqref{s_disjoint_union}
        and the inductive hypothesis we have these equalities: \newpage
        $$\begin{aligned}
           s(L^*)&=s(L_1^*)+s(L_2^*)-1=-s(L_1)+4-2r-s(L_2)-1=\\
           &=2-2r-(s(L_1)+s(L_2)-1)=2-2r-s(L)\:.\\
          \end{aligned}$$        
 \end{itemize}
\end{proof}
Since H-thin links are non split (\cite{Manolescu}) we have that $s(L^*)=-s(L)$ for every $L$ H-thin.
In Section \ref{section:examples} we will see that $L10_{36}^n$ on Thistlethwaite's table is a non split non 
pseudo-thin link for which the previous property is false. 

$\newline$ Checking whether a link is pseudo-thin requires computing filtered Lee homology or at least Khovanov homology. In the
following proposition we give some geometric conditions that ensure that a link belongs to this class. 
\begin{prop}
 \label{prop:conditions}
 Let $(L,{\bf o})$ be a non split link; then if $(L,{\bf o})$ is quasi-alternating or 
 $$\text{lk}({\bf o'},{\bf o})\neq 0\:\:\:\:\:\forall{\bf o'}\in\OO(L)\setminus{\bf \pm o}$$
 we have that $(L,{\bf o})$ is pseudo-thin. 
\end{prop}
\begin{proof}
 In \cite{Manolescu}, Manolescu and Ozsv\'ath show that a quasi-alternating link is \\ H-thin while, in the other case,
 we have $$\text{dim }\Hl^0(L,{\bf o})=2$$ and, as we already know from \cite{Beliakova}, this is enough to conlude that
 $d_{(L,{\bf o})}(0,s)$ is non zero only if $s=s(L,{\bf o})\pm 1$.
\end{proof}
Positive or negative links satisfy the second condition of Proposition \ref{prop:conditions} so that they are pseudo-thin. 
Then we have the following formula for negative links.
\begin{cor}
 If $(D,{\bf o})$ is a negative diagram
 $$\begin{aligned}
   s(D,{\bf o})&=2-2r-s(D^*,{\bf o})=\\
   &=2-2r-(-k(D)+c(D)+1)=k(D)-c(D)+1-2r\\
  \end{aligned}$$
\end{cor}  
where the second equality holds because $k(D^*)=k(D)$ and $c(D^*)=c(D)$. Notice that $r$ is the number of split component of the link 
represented by $D$.
\subsubsection{Connected sum}
Equation \eqref{connected_sum} tells us that $s$ is not additive under connected sum of links, but we can show that this is true for
pseudo-thin links.
\begin{prop}
 \label{prop:connected_sum}
 If $L_1$, $L_2$ and $L_1\sharp L_2$ are pseudo-thin then $s(L_1\sharp L_2)=s(L_1)+s(L_2)$. 
\end{prop}
\begin{proof}
 We already know that $$s(L_1\sharp L_2)\leq s(L_1)+s(L_2)$$ and therefore also that
 $$s(L_1^*\sharp L_2^*)\leq s(L_1^*)+s(L_2^*)\:.$$ 
 Since $L_1 $ and $L_2$ are pseudo-thin then $s(L^*_i)=2-2r_i-s(L_i)$ for every $i$, thus
 $$2-2(r_1+r_2-1)-s(L_1\sharp L_2)\leq 2-2r_1-s(L_1)+2-2r_2-s(L_2)$$ 
 the previous expression becomes
 $$s(L_1\sharp L_2)\geq s(L_1)+s(L_2)$$ and we conclude.
\end{proof}
This proposition applies to knots, for which additivity of connected sum is proved in \cite{Rasmussen}, and
H-thin links.

About the $d$ invariant we have the following formula.
\newpage 
\begin{prop}
 Let $(L_1,{\bf o_1})$ and $(L_2,{\bf o_2})$ be H-thin links then
 $$\begin{aligned}
    d_{(L_1\sharp L_2,{\bf o_1\sharp o_2})}(i,s\pm 1)&=\Biggl|\Biggl\{\bigl((i_1,s_1),(i_2,s_2)\bigr)\in\Z^2\times\Z^2\:|\\
    &|\:i_1+i_2=i, s_1+s_2=s\text{ and }d_{L_j}(i_j,s_j\pm 1)>0\:\forall j\Biggr\}\Biggr|\\
   \end{aligned}$$
 and $0$ otherwise.   
\end{prop}
\begin{proof}
 First we recall that a H-thin link $L$ has $d_{L}$ invariant supported in $s=2i+s(L)\pm 1$. There are $2^{n-1}$ pairs
 $\bigl(i,s\pm 1=2i+s(L)\pm 1\bigr)$ counted with multiplicity, where $n$ is the number of components of $L$, and 
 each one increases by one the value of $d_{L}(i,s\pm 1)$.
 $\newline$ Let $(i_1,s_1\pm 1)$ be such a pair for $L_1$ and $(i_2,s_2\pm 1)$ be another one for $L_2$: this means that there exist
 ${\bf o_1'}$ and ${\bf o_2'}$ orientations of $L_1$ and $L_2$ such that $2\cdot\text{lk}({\bf o_j'},{\bf o_j})=i_j$ for 
 every $j$. Now we have that the orientation ${\bf o_1'\sharp o_2'}$ of $L_1\sharp L_2$ gives 
 $$i=2\cdot\text{lk}({\bf o_1'\sharp o_2'},{\bf o_1\sharp o_2})=2\cdot\text{lk}({\bf o_1'},{\bf o_1})+2\cdot\text{lk}({\bf o_2'},{\bf o_2})=i_1+i_2$$
 and it corresponds to the pair
 $$\bigl(i=i_1+i_2,s\pm 1=2i+s(L_1\sharp L_2)\pm 1=2i_1+2i_2+s(L_1)+s(L_2)\pm 1=s_1+s_2\pm 1\bigr)$$
 where we have used Proposition \ref{prop:connected_sum}. Finally, we note that all of the $2^{n_1+n_2-2}$ pairs of 
 $L_1\sharp L_2$ are obtained in this way, in fact $L_1\sharp L_2$ has $n_1+n_2-1$ components.
 $\newline$ The proof is completed.
\end{proof}
This result gives the following theorem.
\begin{teo}
 If $L_1$,$L_2$ and $L_1\sharp L_2$ are pseudo-thin oriented links then the value of $s(L_1\sharp L_2)$ is independent from 
 the choice of components used to perform the connected sum. 
 $\newline$ If, in addition, $L_1$ and $L_2$ are H-thin then this is true for $d_{L_1\sharp L_2}$ as well.
\end{teo}

\subsection{Improving Lobb's inequality}
Andrew Lobb in \cite{Lobb}, working on Rasmussen invariant, found an upper and a lower bound for the value of $s(D,{\bf o})$, 
which depend
from the link diagram $D$. Now we want to show that, for pseudo-thin links, we can give a better lower bound.
This also allows us to give a better lower bound for slice genus of links.
\begin{defin}
Let $(D,{\bf o})$ be an oriented diagram of a link. Then we define $U(D,{\bf o})$ and $V(D,{\bf o})$ as
$$U(D,{\bf o})=k(D,{\bf o})+w(D,{\bf o})-2S^-(D,{\bf o})+1$$
$$V(D,{\bf o})=-k(D,{\bf o})+w(D,{\bf o})+2S^+(D,{\bf o})-1$$
where
 \begin{itemize}
  \item $k(D,{\bf o})$ is the number of circles of the coherent resolution of $(D,{\bf o})$.
  \item $w(D,{\bf o})$ is the writhe of the diagram.
  \item $S^{\pm}(D,{\bf o})$ is the number of connected components of the graph obtained in the following way: for each circle
        of the coherent resolution of $(D,{\bf o})$ we have a vertex, and we connect two vertices with an edge if and only if they correspond to
        circles connected by at least a positive (negative) crossing.
 \end{itemize}
\end{defin}
The following figure shows the computation of $U$ and $V$ for the figure eight knot. 
Lobb's inequalities are formulated in terms of $U$ and $V$.
\begin{figure}[H]
  \centering
  \includegraphics[width=0.6\textwidth]{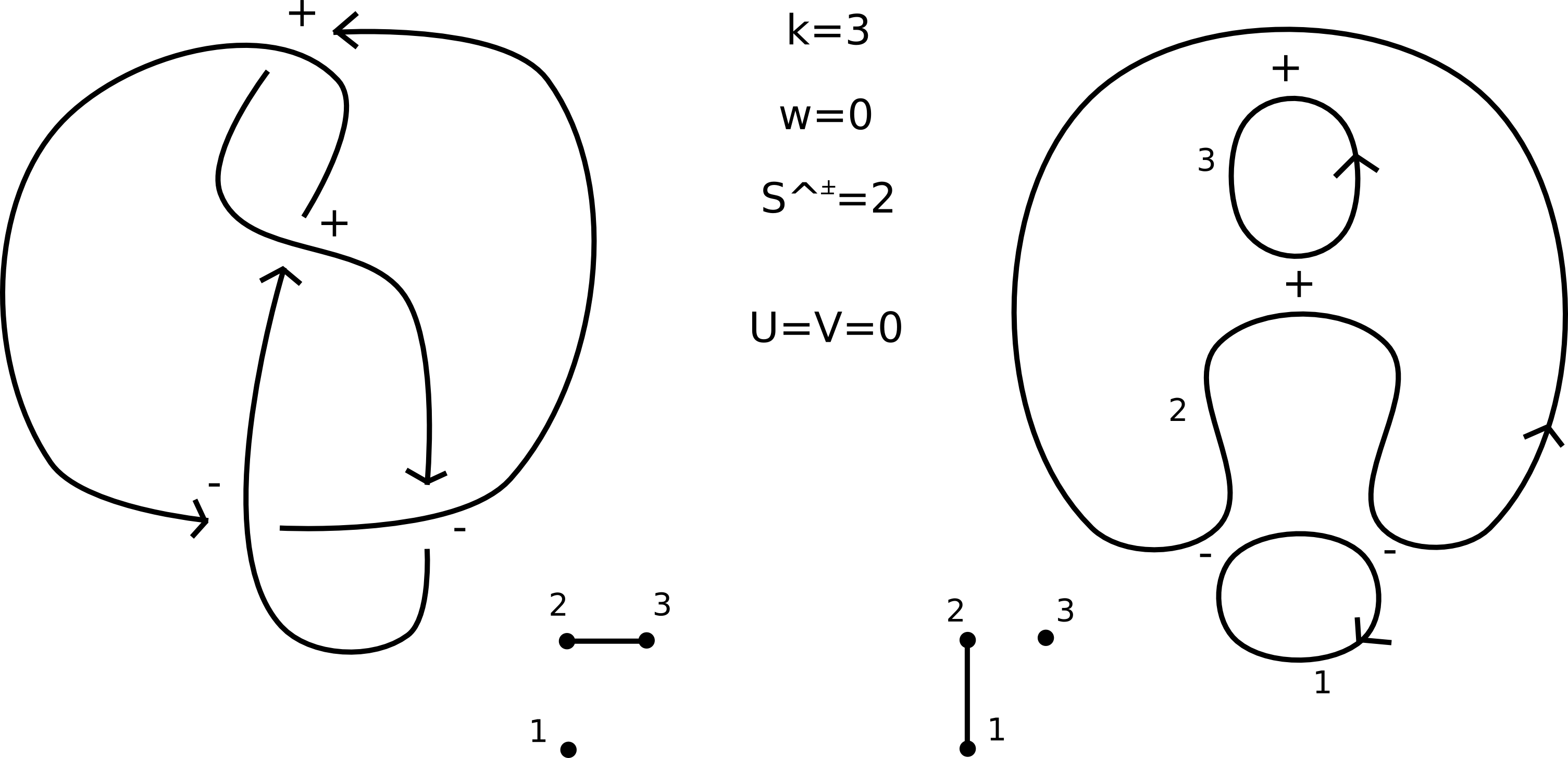} 
  \caption{}
\end{figure}
\begin{teo}
 For every oriented diagram of a link is 
 \begin{equation}
  \label{sup_s}
  U(D,{\bf o})\geq s(D,{\bf o})
 \end{equation}
 and
 \begin{equation*}
  s(D,{\bf o})\geq 2-2n+V(D,{\bf o})
 \end{equation*}
  where $n$ is the number of components of the link.
\end{teo}
\begin{proof}
The first inequality is proved in \cite{Lobb}.
Now, it is trivial that $$V(D,{\bf o})=-U(D^*,{\bf o})$$ then
$$U(D^*,{\bf o})\geq s(D^*,{\bf o})$$ and from Equation \eqref{mirror}
$$-V(D,{\bf o})\geq 2-2n-s(D,{\bf o})$$ that is equivalent to 
$$s(D,{\bf o})\geq 2-2n+V(D,{\bf o})$$
which proves the lower bound in \cite{Lobb}.
\end{proof}
Our result improves the second inequality in terms of the number $r$ of split components of the link.
\begin{teo}
 \label{teo:lower_bound}
 If $(D,{\bf o})$ is the diagram of a pseudo-thin link, then Inequality \eqref{inf_s} $$s(D,{\bf o})\geq 2-2r+V(D,{\bf o})$$ 
 holds.
\end{teo}
\begin{proof}
 It is enough to observe that
 $$U(D^*,{\bf o})=-V(D,{\bf o})\geq s(D^*,{\bf o})=2-2r-s(D,{\bf o})$$ where the last equality follows from \ref{prop:mirror_pseudo}.
\end{proof}
In addition, for non split alternating links, these inequalities are indeed equalities. The 
proof is identical
to the one in \cite{Lobb} for knots.
\begin{prop}
 If $(D,{\bf o})$ is a non split alternating diagram then \\ $U(D,{\bf o})=V(D,{\bf o})$ and consequently, by Theorem 
 \ref{teo:lower_bound}, $U(D,{\bf o})=s(D,{\bf o})=V(D,{\bf o})$, that is
 $$s(D,{\bf o})=-\sigma(D,{\bf o})=-k(D,{\bf o})+w(D,{\bf o})+2S^+(D,{\bf o})-1\:.$$
\end{prop}
Theorem \ref{teo:lower_bound} gives a better lower bound for $s$ than Lobb's one, especially for non split links with a high 
number of components.

\section{Computations of the slice genus}  
\label{section:four}
\subsection{Morse moves and a connection between \texorpdfstring{$s(L)$}{Lg} and \texorpdfstring{$g_*(L)$}{Lg}}
We can say when two diagrams represent weakly cobordant links.
Given $\Sigma$ a weak cobordism between links $L$ and $L'$, from Morse theory we know 
that $\Sigma$ can be decomposed in a finite number of elementary cobordisms.
\begin{figure}[H]
  \centering
  \includegraphics[width=0.52\textwidth]{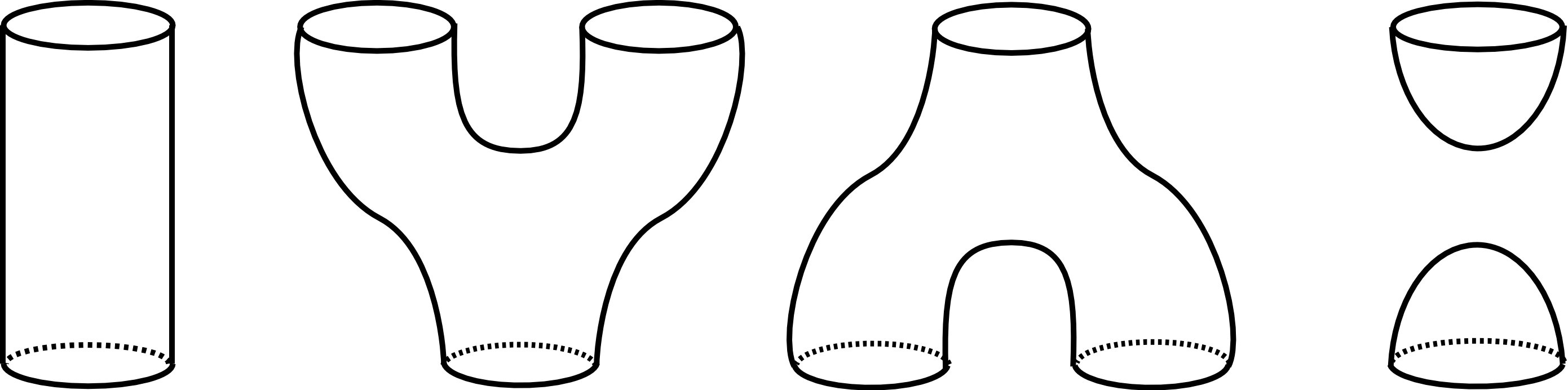} 
  \caption{}
\end{figure}
This means that if we have $D$ and $D'$ diagrams of $L$ and $L'$, there is a weak cobordism between them if and only if
$D'$ is obtained from $D$ with a finite number of Reidemeister and Morse moves. Each Morse move corresponds to a 
$k$-handle in the cobordism decomposition.
\begin{figure}[H]
  \centering
  \includegraphics[width=0.4\textwidth]{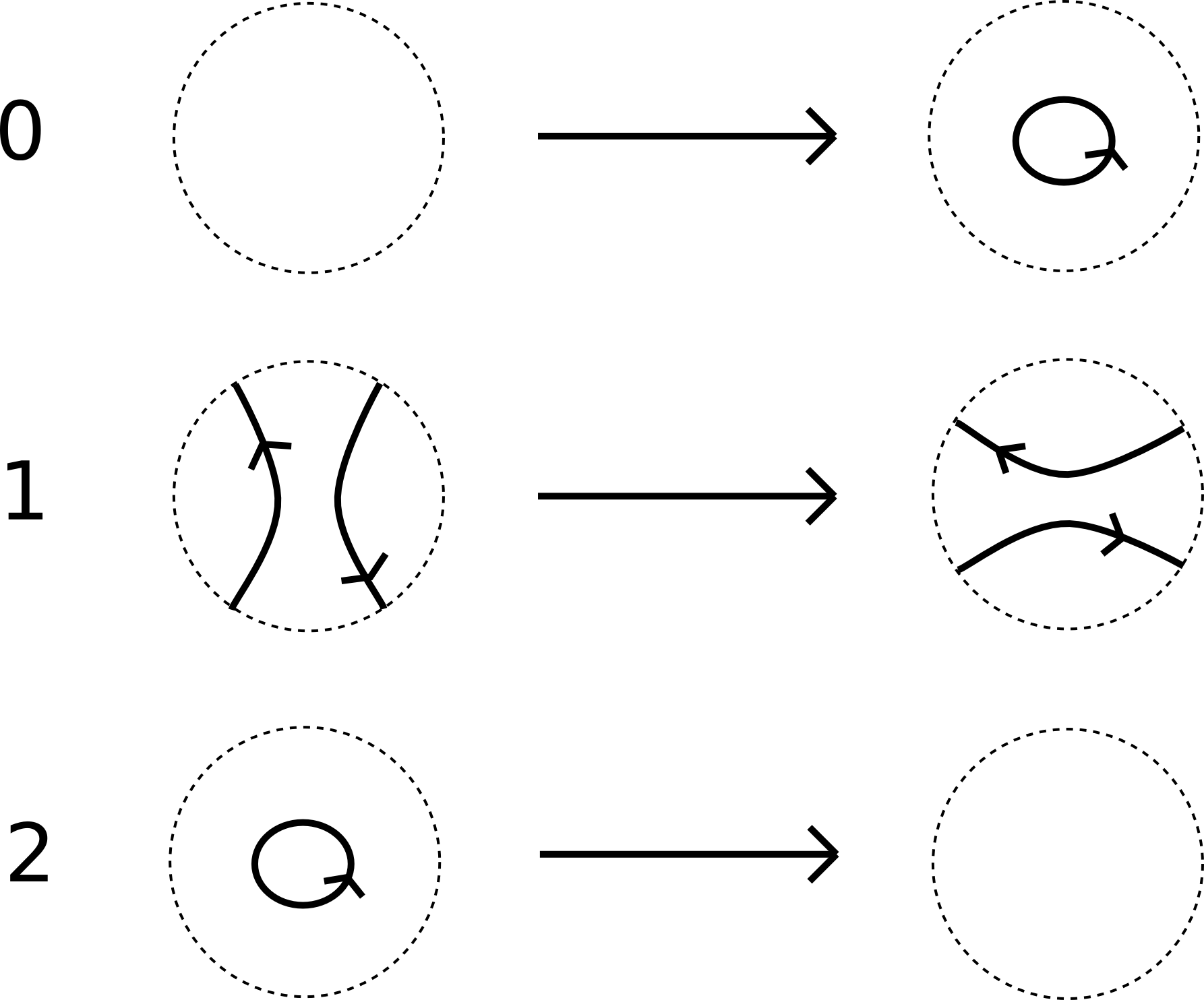} 
  \caption{Local behaviour of Morse moves}
\end{figure}  
Now we use Equation ~\eqref{cobordismo_s} to show that the $s$ invariant gives a lower bound for the slice genus of a link.
In fact, given $\Sigma$ a weak cobordism between a $n$ component link $L$ and the unknot, we have 
$$|s(L)|\leq -\chi(\Sigma)=2g(\Sigma)+n-1$$
from which we obtain the following equation: 
\begin{equation}
 \label{bound_inf_s}
 g_*(L)\geq\dfrac{|s(L)|+1-n}{2}\:.
\end{equation}
An upper bound for $g_*(L)$ is easy to find using the following proposition.
\begin{prop}
 Given a link $L$, then the Seifert algorithm associated to a diagram $(D,{\bf o})$ of $L$, gives a Seifert surface $S$ for 
 $L$ such that $\chi(S)=k(D,{\bf o})-c(D)$.
\end{prop}
\begin{proof}
 It simply follows from the construction of the surface in the Seifert algorithm.   
\end{proof}
Thus
\begin{equation}
 \label{bound_sup_s}
 \dfrac{2-n-k(D,{\bf o})+c(D)}{2}=\dfrac{2-n-\chi(S)}{2}=g(S)\geq g(L)\geq g_*(L)\:.
\end{equation}
We want to apply what we have shown to some families of links in order to compute $g_*$.

\subsection{The Milnor conjecture}
For a long time this important conjecture about the slice genus of torus knots was unsolved. Later, 
in the early nineties, Kronheimer and Mrowka proved it using gauge theory. One of the most interesting result of Rasmussen
in \cite{Rasmussen} consists precisely in providing a purely combinatorial proof of the Milnor conjecture, 
with his own invariant. In this paper we use the same technique to prove Proposition \ref{prop:Milnor}.
$\newline$ Since a torus link $T_{p,q}$, with all components oriented in the same direction, is positive we can find the value 
of the $s$ invariant for every $(p,q)$.
The coherent resolution of $D_{p,q}$, the standard diagram of $T(p,q)$, gives:
\begin{figure}[H]
  \centering
  \includegraphics[width=0.65\textwidth]{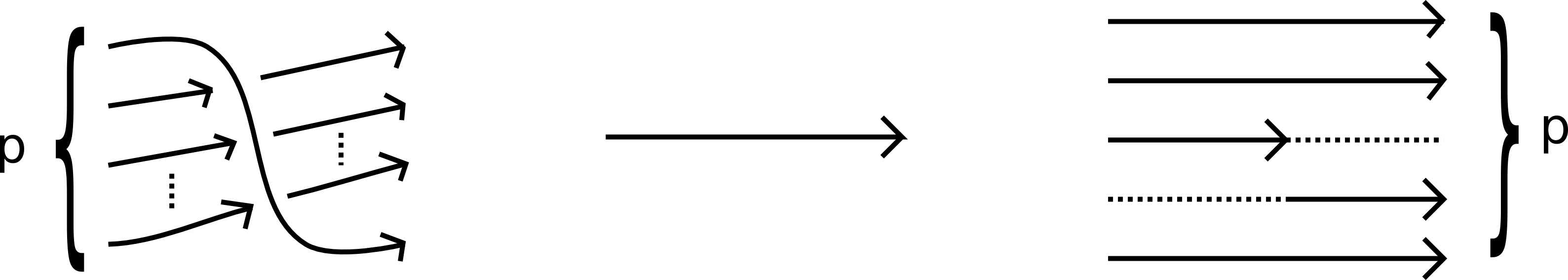}
  \caption{}
\end{figure}
Then $k(D_{p,q})=p$ and $c(D_{p,q})=q(p-1)$, whence Proposition \ref{prop:positive} says   
$$s(T_{p,q})=-p+q(p-1)+1=(p-1)(q-1)\:.$$
Now Equation \eqref{bound_inf_s} says that $$g_*(T_{p,q})\geq\dfrac{(p-1)(q-1)+1-n}{2}$$
instead Equation \eqref{bound_sup_s} 
gives $$g_*(T_{p,q})\leq\dfrac{2-n-p+q(p-1)}{2}=\dfrac{(p-1)(q-1)+1-n}{2}$$
so we have found that Equality \eqref{Milnor} $$g_*(T_{p,q})=\dfrac{(p-1)(q-1)+1-n}{2}$$
holds for every $p,q$.

Taking $n=1$, we have the same formula of the original conjecture. We observe that the only non trivial slice torus link is 
the Hopf link: 
in fact, we should have $(p-1)(q-1)=n-1$, so $p-1$ would be a divisor of $n-1$, whence $p\leq n$, but $n$ is the $\text{gcd}$ of 
$p$ and $q$, thus $n\leq p$ and then $p=n$. The equation becomes
$$(n-1)(q-1)=n-1$$ which, dividing by $n-1$ (because no non trivial torus knot is slice) gives $q-1=1$. The same holds for $p$, and
then we have $p=q=2$ which corresponds exactly to the Hopf link.
$\newline$ Finally, since $T_{p,q}$ is a positive link, we can say that every torus link is pseudo-thin thanks to
\ref{prop:conditions}: in fact, every orientation other than the reverse has $\text{lk}({\bf o'},{\bf o})\neq 0$.
This means that all the results we obtained on torus links hold even for their mirrors, which are negative links.
\subsection{Pretzel links}
In this subsection we prove Proposition \ref{prop:boh}. First we make some general 
considerations: these links have always two components $L_1$ and $L_2$, so they have two relative orientations. It follows 
that 
$\text{lk}({\bf o'},{\bf o})=\text{lk}(L_1,L_2)=h-k$. If $h\neq k$, $P_{2h,2k,2l+1}$ is pseudo-thin thanks to 
\ref{prop:conditions}:
they are non split because their determinant is non zero.

When $h$ or $k$ are zero, links are either trivial or torus links and we ignore them; also, we do not consider the case $h>0$,
$k>0$ that we will study later. Hence there are three cases remaining.

$\newline$ {\bf a)} $h<0$ , $k<0$ and $h\neq k$
$\newline$ We can use Inequalities \eqref{inf_s} and \eqref{sup_s} to obtain that $$s(P_{2h,2k,2l+1})=2l+2h+2\pm 1$$
as we can see in the following figure.
\begin{figure}[H]
        \centering
        \includegraphics[width=0.52\textwidth]{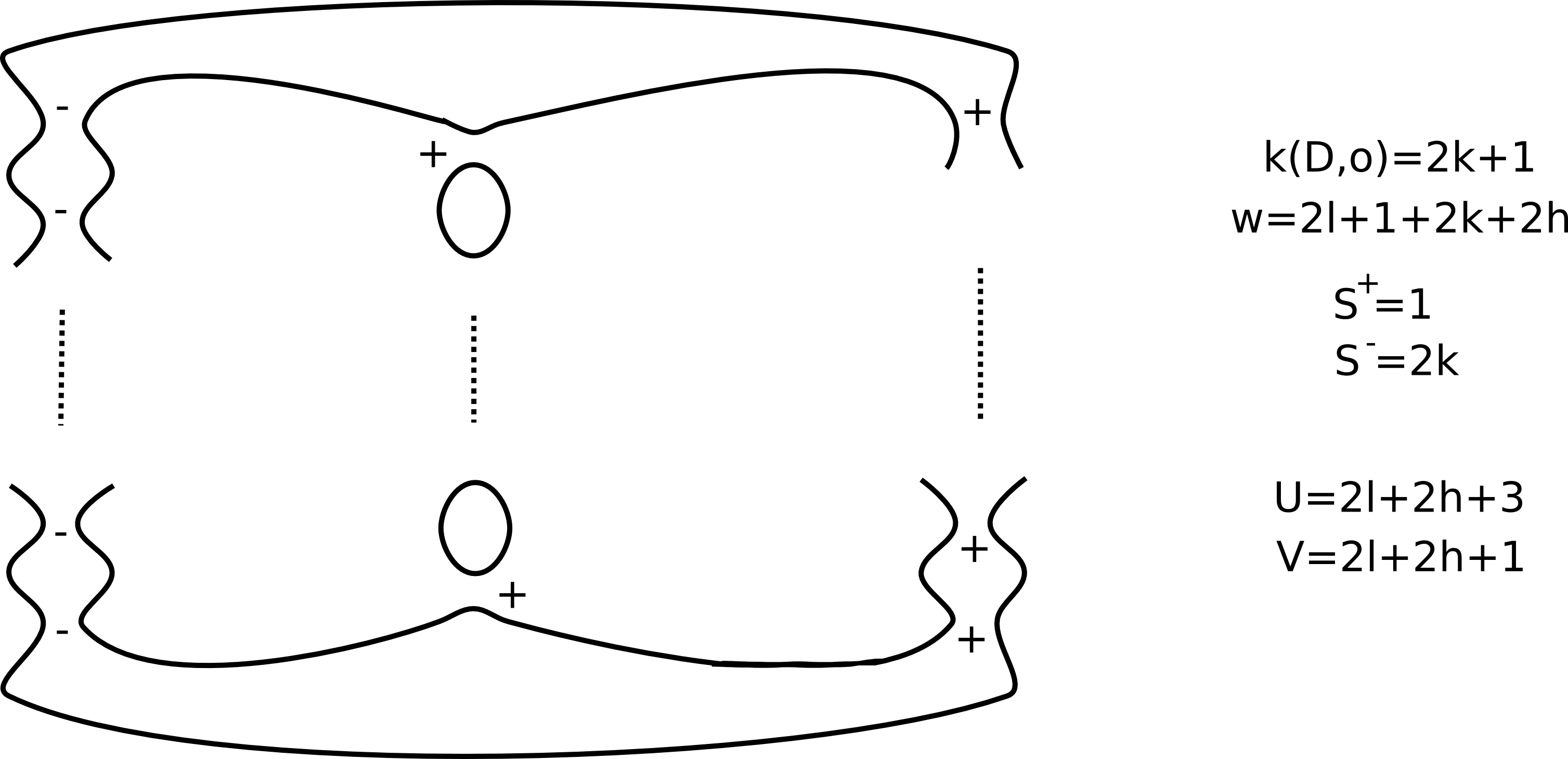}  
        \caption{}
\end{figure}
While, for the other relative orientation, the same argument says that $$s(P_{2h,2k,2l+1})=2l+2k+2\pm 1\:.$$
Now we take a look at Figure ~\ref{Pretzel}: if we apply a type 1 Morse move to two arcs on the central strand which belong 
to different components, we obtain that $P_{2h,2k,2l+1}$ is cobordant to the following knot.
\begin{figure}[H]
        \centering
        \includegraphics[width=0.3\textwidth]{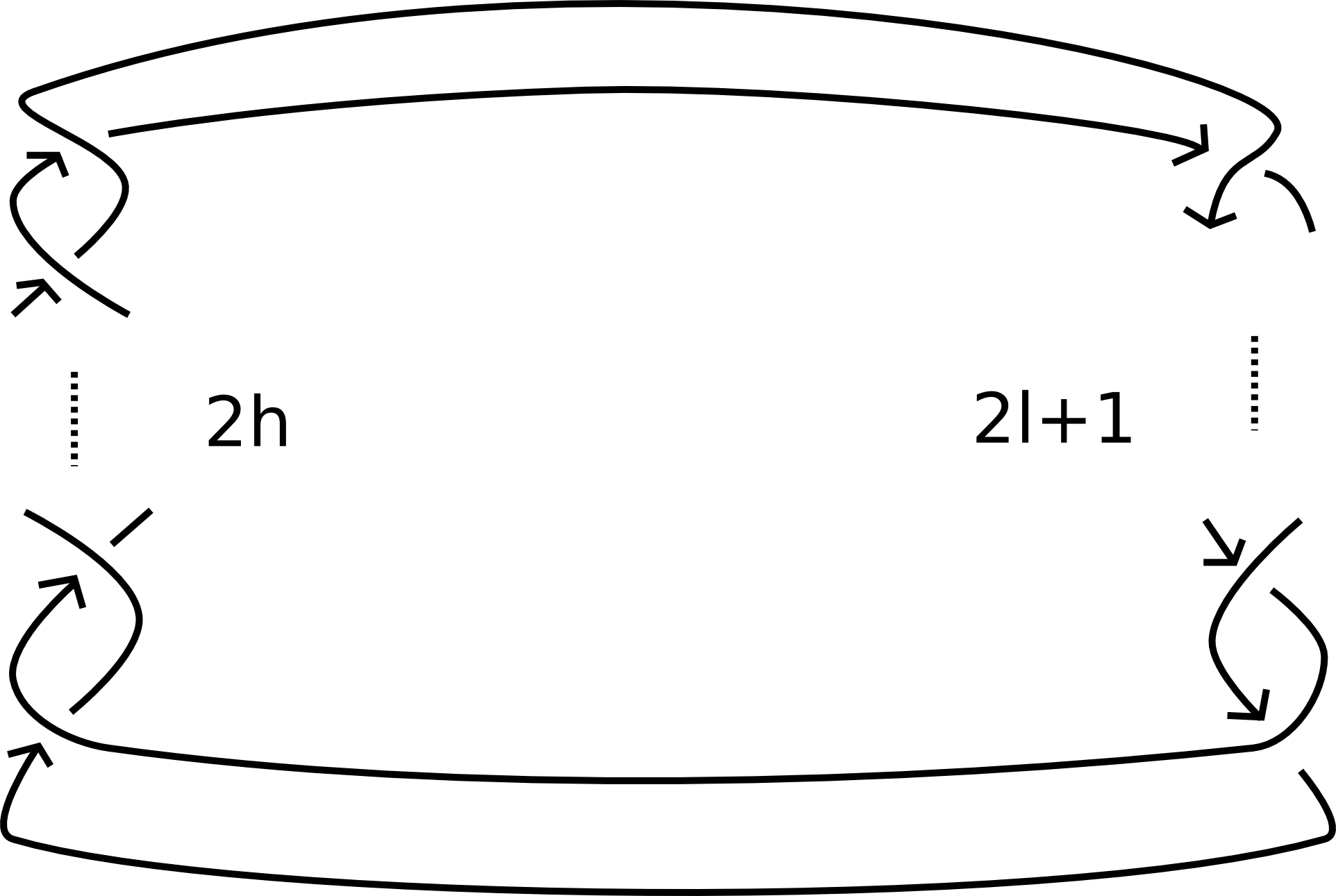} 
\end{figure}
If we set the condition that $l+h\geq 0$, that is $2l+1\geq|2h|$, then, with some Reidemeister moves, the $|2h|$ 
crossings in the left strand can be canceled out with $|2h|$ of the $2l+1$ crossings in the right strand. Then we have that 
$P_{2h,2k,2l+1}$ is cobordant to the torus knot $T_{2,2l+1+2h}$.

Now we make two observations: the first is that the cobordism we found consists of a saddle which connects two
components (so the genus does not change); the second one is that $g_*(T_{2,2l+1+2h})=l+h$. From these facts we have
       $$g_*(P_{2h,2k,2l+1})\leq l+h\:.$$
If $s(P_{2h,2k,2l+1})=2l+2h+3$ then ~\eqref{bound_inf_s} would give $g_*(P_{2h,2k,2l+1})\geq l+h+1$ and this would be a 
contradiction; so $$s(P_{2h,2k,2l+1})=2l+2h+1\:\:\:\:\:\text{ and }\:\:\:\:\:g_*(P_{2h,2k,2l+1})=l+h\:.$$
If $l+h=-1$ the knot of the figure is the unknot so we can directly deduce that $g_*(P_{2h,2k,2l+1})=0$.
       
With the other relative orientation we work similarly, except that the Morse move is done on the left strand. 
These results follow as in the previous case: \\
if $l+k\geq 0$ then $$s(P_{2h,2k,2l+1})=2l+2k+1\:\:\:\:\:\text{ and }\:\:\:\:\:g_*(P_{2h,2k,2l+1})=l+k\:.$$
If $l+k=-1$ then $g_*(P_{2h,2k,2l+1})=0$.

$\newline$ {\bf b)} $h>0$ and $k<0$
$\newline$ The link is positive so, thanks to ~\ref{prop:positive}, we know that 
       $$s(P_{2h,2k,2l+1})=2l+2h+1\:.$$ 
If we change the relative orientation then we must use estimates: 
\begin{figure}[H]
        \centering
        \includegraphics[width=0.5\textwidth]{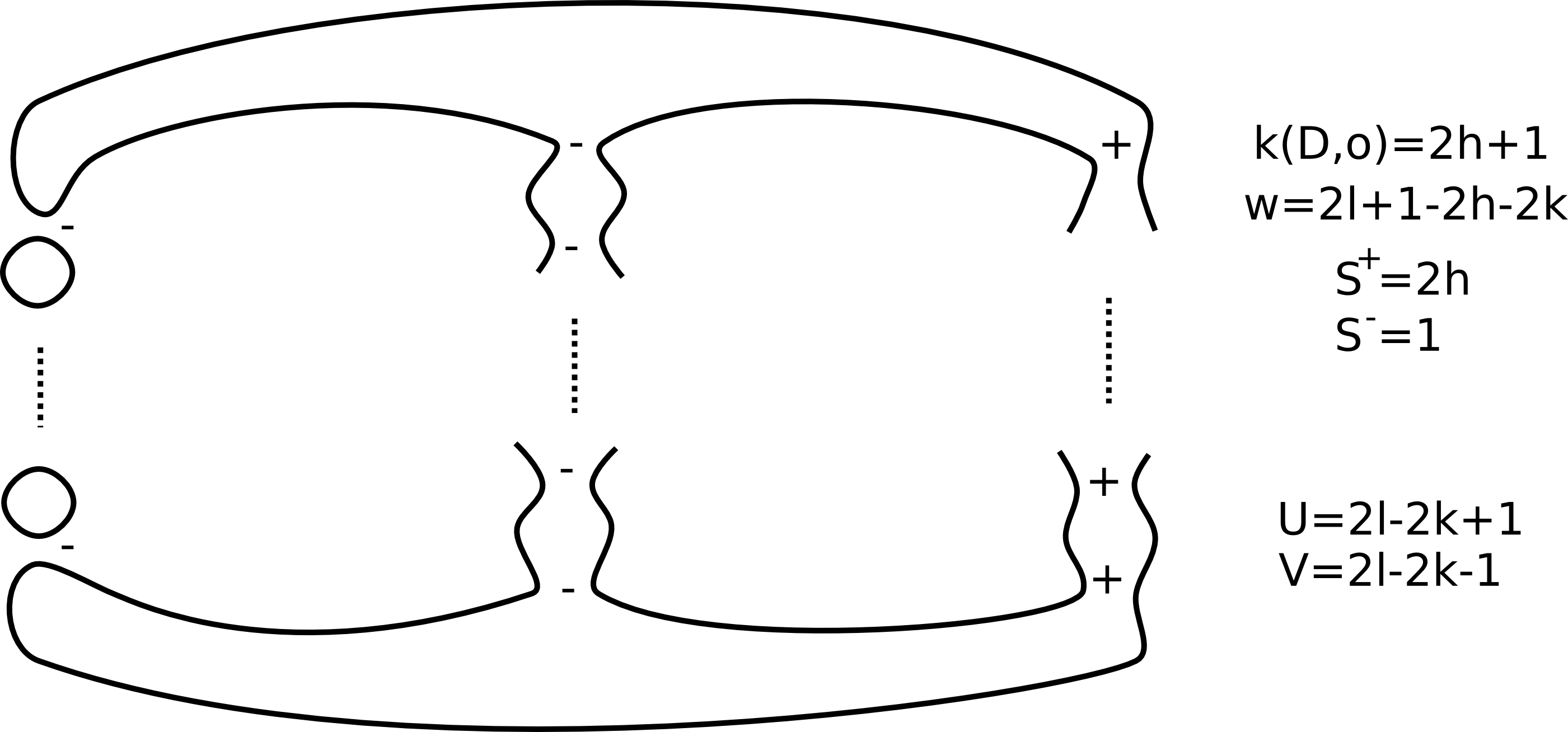}   
        \caption{}
\end{figure} 
the $s$ invariant can be $2l+2k\pm 1$.
$\newline$ Conclusions follow exactly like before: for the standard orientation we have $$g_*(P_{2h,2k,2l+1})=l+h\:.$$       
For the other one, if $l+k<0$ then $$s(P_{2h,2k,2l+1})=2l+2k+1\:\:\:\:\:\text{ and }\:\:\:\:\:g_*(P_{2h,2k,2l+1})=-k-l-1\:.$$
If $l+k=0$ then $g_*(P_{2h,2k,2l+1})=0$.
  
$\newline$ {\bf c)} $h<0$ and $k>0$
$\newline$ This case is formally the same as the previous one so we only report the results: 
$$s(P_{2h,2k,2l+1})=2l+2h\pm 1$$ 
if $l+h<0$ then $$s(P_{2h,2k,2l+1})=2l+2h+1\:\:\:\:\:\text{ and }\:\:\:\:\:g_*(P_{2h,2k,2l+1})=-l-h-1\:.$$
If $l+h=0$ then $g_*(P_{2h,2k,2l+1})=0$.

While, with the other relative orientation, the link is positive and then we have $$s(P_{2h,2k,2l+1})=2l+2k+1$$       
hence it is always $g_*(P_{2h,2k,2l+1})=l+k$.
\subsection{Linked twist knots}
We have already defined $Tw_n$ links in the introduction. The linking number between their components is always 2 so 
they are non split pseudo-thin links.
Figure \ref{asd} shows the coherent resolution of the diagram in Figure \ref{Tw_n}. 

We obtain that $s(Tw_n)=3$ for every $n\geq 0$ and, using Equation \eqref{bound_inf_s} and
keeping in mind that links have two components, we also have $g_*(Tw_n)\geq 1$.

We see that no link of this family is slice. However, if $n=1$, one component is a Stevedore knot which is known to be 
slice; taking advantage of this we can prove that $g_*(Tw_1)=1$. \newpage
\begin{figure}[H]
 \centering
 \includegraphics[width=0.75\textwidth]{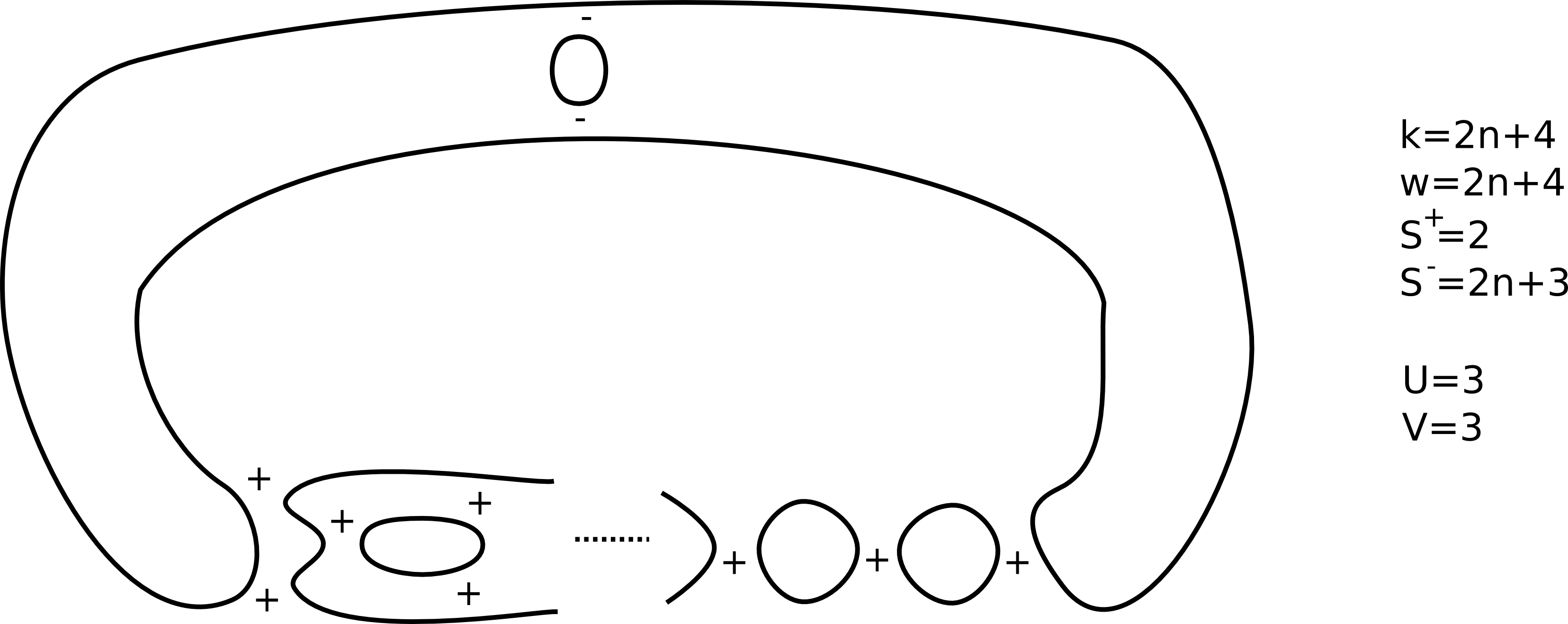} 
 \caption{}
 \label{asd}
\end{figure} 
All we need is to find a genus 1 surface properly embedded in $D^4$ which has a $Tw_1$ link ($L10^n_{19}$ in 
Thistlethwaite's table) as boundary.
\begin{figure}[H]
 \centering
 \includegraphics[width=0.75\textwidth]{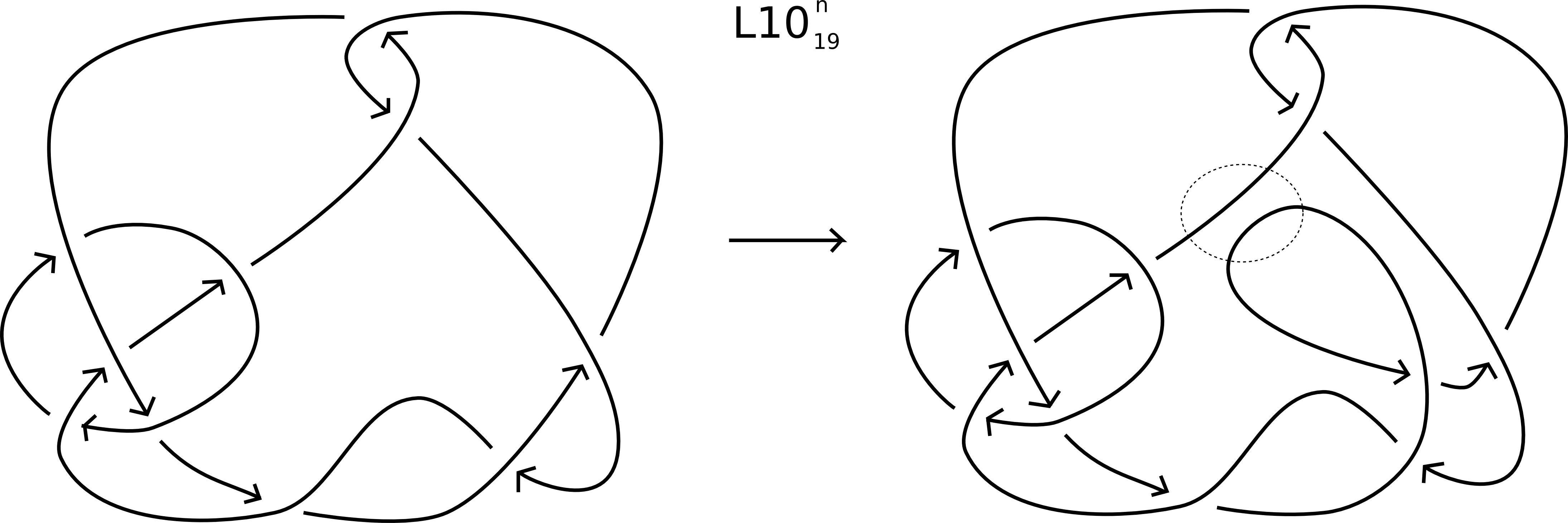} 
\end{figure}
Now we apply a type 1 Morse move in the highlighted tangle of previous figure. 
\begin{figure}[H]
 \centering
 \includegraphics[width=0.55\textwidth]{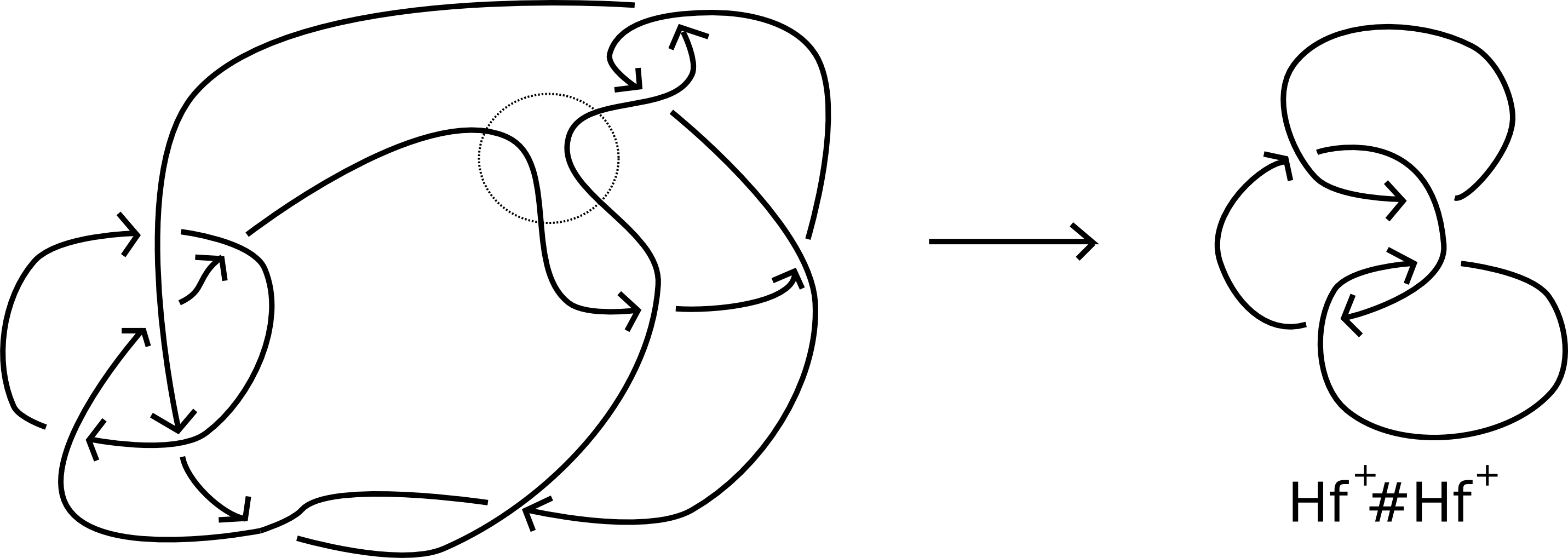}  
\end{figure}
We have decribed a cobordism between $L10^n_{19}$ link and the connected sum of two positive Hopf links which has genus 0.
This means that the slice genus is 0 as well, and we can outline the cobordism in this way.
\begin{figure}[H]
 \centering
 \includegraphics[width=0.4\textwidth]{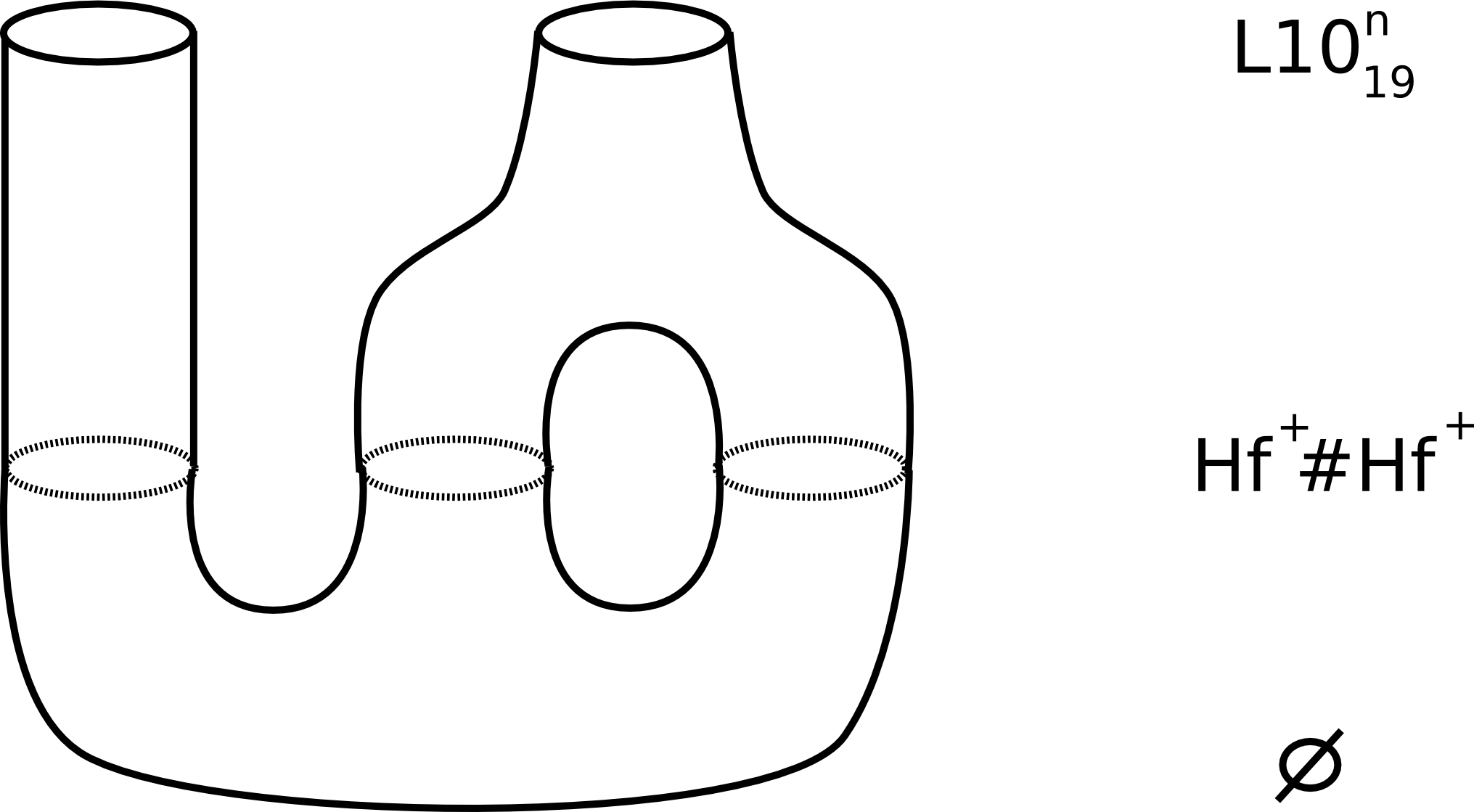}   
 \caption{The surface and its boundary are actually knotted in $D^4$}
\end{figure}
To sum up, we proved that $L10^n_{19}$ link is the boundary of a torus in $D^4$.

\section{Results on strong cobordisms}
\label{section:strong}
\subsection{Strong slice genus}
In this paragraph we want to define the invariant $g^*_*$, as mentioned in the introduction. 
\begin{defin}
 Given $L$ a $n$ component algebrically split link (which means $\text{lk}(L_i,L_j)=0$ for every $i,j$) we call the strong slice genus of $L$ the  
 minimum genus of a strong 
 cobordism between $L$ and $\bigsqcup^n\bigcirc$. We denote this number with $g_*^*(L)$.
\end{defin}
This invariant is well defined: there is always at least one strong cobordism as in the definition. To see this we prove the 
following lemma.
\begin{lemma}
 If $L$ is a $n$ component algebrically split link
 then each component bounds a compact, orientable surface in $D^4$ and these surfaces are all disjoint.
\end{lemma}
\begin{proof}
Let $S_i$ be Seifert surfaces in $D^4$ for each component of $L$. We can suppose $S_i$ to be tranverse, so they intersect only 
in a finite number of points. 
Each pair of surfaces $S_i$ and $S_j$ intersect in an even number of points, one half positive
and one half negative; this is because the sum of all signs is equal to $\text{lk}(L_i,L_j)$, which is zero. Let $x_+$ and 
$x_-$ be two points such that $x_+,x_-\in S_i\cap S_j$: we can cancel these two intersections by adding a tube on $S_i$ 
disjoint from the other surfaces. Since $S_i$ is orientable and $x_{\pm}$ have opposite signs, the new surface is still orientable.

Finally, after removing all the intersections from every pair of surfaces, we obtain what we wanted. The reader can observe 
that, in the process, the genus of $S_i$ has been increased a lot.
\end{proof}
We have that $g_*(L)\leq g_*^*(L)$, because a strong cobordism is indeed a weak cobordism, and $g_*^*(L)=0$ if and only if $L$ is strongly slice, that is, strongly concordant to 
the unlink or, in other words, each component bounds a disk in $D^4$ and all these disks are disjoint.

$\newline$ We conclude with the following proposition, which states some basic properties of $g^*_*$.
\begin{prop}
 \label{prop:orientations}
 Let $L$ be a $n$ component algebrically split link, then 
 \begin{itemize}
  \item If $L'$ is obtained from $L$ reversing the orientation of one component then
        $g^*_*(L)=g^*_*(L')$.
  \item $g^*_*(L)=g^*_*(L^*)$.          
 \end{itemize}
\end{prop}
\begin{proof}
 Both statements are obvious:
 \begin{itemize}
  \item If $\Sigma$ is the minimum strong cobordism between $L$ and $\bigsqcup^n\bigcirc$, then $\Sigma'$, obtained 
        reversing the orientation of one component of $\Sigma$, is the minimum strong cobordism between $L'$ and the unlink.
  \item Suppose $D$ is a diagram of $L$, then there is a sequence of Reidemeister and Morse moves that take
        $D$ into $\bigsqcup^n\bigcirc$. Then the same sequence, mirrored, take $D^*$ into $\bigsqcup^n\bigcirc$.
        This means that we have a strong cobordism of the same genus for 
        $L^*$.        
 \end{itemize}
\end{proof}

\subsection{Applications to pseudo-thin links}
Now we prove Theorem \ref{teo:not_concordant}; the proof uses the same techniques of Pardon applied to all 
pseudo-thin links, not only H-thin links as in \cite{Pardon}. First, we call 
diameter of a link $L$ the value 
$$\max\{s\:|\: d_L(0,s)>0\}-\min\{s\:|\: d_L(0,s)>0\}$$ Then the diameter of $M_i$ 
 is at least 2 for every $i$,  
 from Property \eqref{d_support}, so Equation \eqref{disjoint_union} says that $M$ has diameter at least $2k$. 
 Since a non split pseudo-thin link has always diameter equal to 2, we have that the diameter of $L$ is $2h$, for the same 
 reason.
 
 The cobordism $\Sigma$ and its inverse induce the following maps in homology:
 $$\Hl^0(L)\xlongrightarrow{F^*_{\Sigma}}\Hl^0(M)\xlongrightarrow{F^*_{-\Sigma}}\Hl^0(L)$$
 both are filtered isomorphisms of degree $-2g(\Sigma)$ for ~\ref{prop:filtered_homology}.
 
 This says that filtered Lee homology of $M$ in 0 is supported in degrees between 
 $$\max\{s\:|\: d_L(0,s)>0\}+2g(\Sigma)\text{ and }\min\{s\:|\: d_L(0,s)>0\}-2g(\Sigma)$$ 
 this means that $4g(\Sigma)+2h$ should be at least equal to the diameter of $M$ and so $4g(\Sigma)+2h\geq 2k$. 
 Thus we conclude that 
 $$g(\Sigma)\geq\left\lceil\dfrac{k-h}{2}\right\rceil\geq\dfrac{k-h}{2}\:.$$
We have the following corollary.
\begin{cor}
 If $L$ and $M$ are strongly concordant links with $h$ and $k$ split conponents and $L$ is pseudo-thin then $h\geq k$.
 $\newline$ In particular 
 \begin{itemize}
  \item Every non split pseudo-thin link is not strongly concordant to a split link. 
  \item A pseudo-thin link is strongly slice if and only if it is a disjoint union of slice knots.
 \end{itemize}
\end{cor}
\begin{proof}
 Suppose that $k\geq h+1$, then from the previous theorem, we have that a strong cobordism between $L$ and $M$ has
 $g(\Sigma)\geq 1$; this is a contradiction because the links are strongly concordant.
 \begin{itemize}
  \item If $L$ is a non split pseudo-thin link then every link strongly concordant to it is also non split.
  \item If $h=1$, from the previous statement it follows that $M$ is a slice knot and the converse is obvious. For $h>1$ we reason 
        in the same way.
 \end{itemize}
\end{proof}
The most important application of Theorem \ref{teo:not_concordant} is found by using the techniques of its proof to prove the 
following estimate, 
which is a version of ~\eqref{bound_inf_s} for the strong slice genus.
\begin{teo}
 \label{teo:mine}
 Let $L=L_1\cup...\cup L_n$ be a $n$ component non split pseudo-thin link which is also algebrically split;
 then we have Inequality \eqref{boh}:
 $$g_*^*(L)\geq\dfrac{|s(L)|+n-1}{2}\:.$$
\end{teo}
\begin{proof}
 First we recall that a non split pseudo-thin link has $d_L(0,s)>0$ only for $s=s(L)\pm 1$.
 
 Let $\Sigma$ and $M=\bigsqcup^n\bigcirc$ be as in Theorem \ref{teo:not_concordant} and let $F^*_{\pm\Sigma}$ be the same 
 homology maps as in its proof. We already know the values of $d_{\bigsqcup^n\bigcirc}$ and then we have 
 $$\max\{s\:|\:d_{\bigsqcup^n\bigcirc}>0\}=n\text{ and }\min\{s\:|\:d_{\bigsqcup^n\bigcirc}>0\}=-n\:.$$
 Since these maps are filtered of degree $-2g(\Sigma)$ we have 
 $$s(L)+1+2g(\Sigma)\geq n\:\:\:\text{ and }\:\:\:s(L)-1-2g(\Sigma)\leq -n$$ that means
 $g_*^*(L)\geq\dfrac{\pm s(L)+n-1}{2}$
 which is the statement. 
\end{proof}

\section{Examples}
\label{section:examples}
We talked about strongly slice links in the previous section, but we have not yet shown any non split one. 
In fact, it is not so easy to find an example of such a link. 
Let us take a look at Figure \ref{L10n36}: \newpage
\begin{figure}[H]
  \centering
  \includegraphics[width=0.25\textwidth]{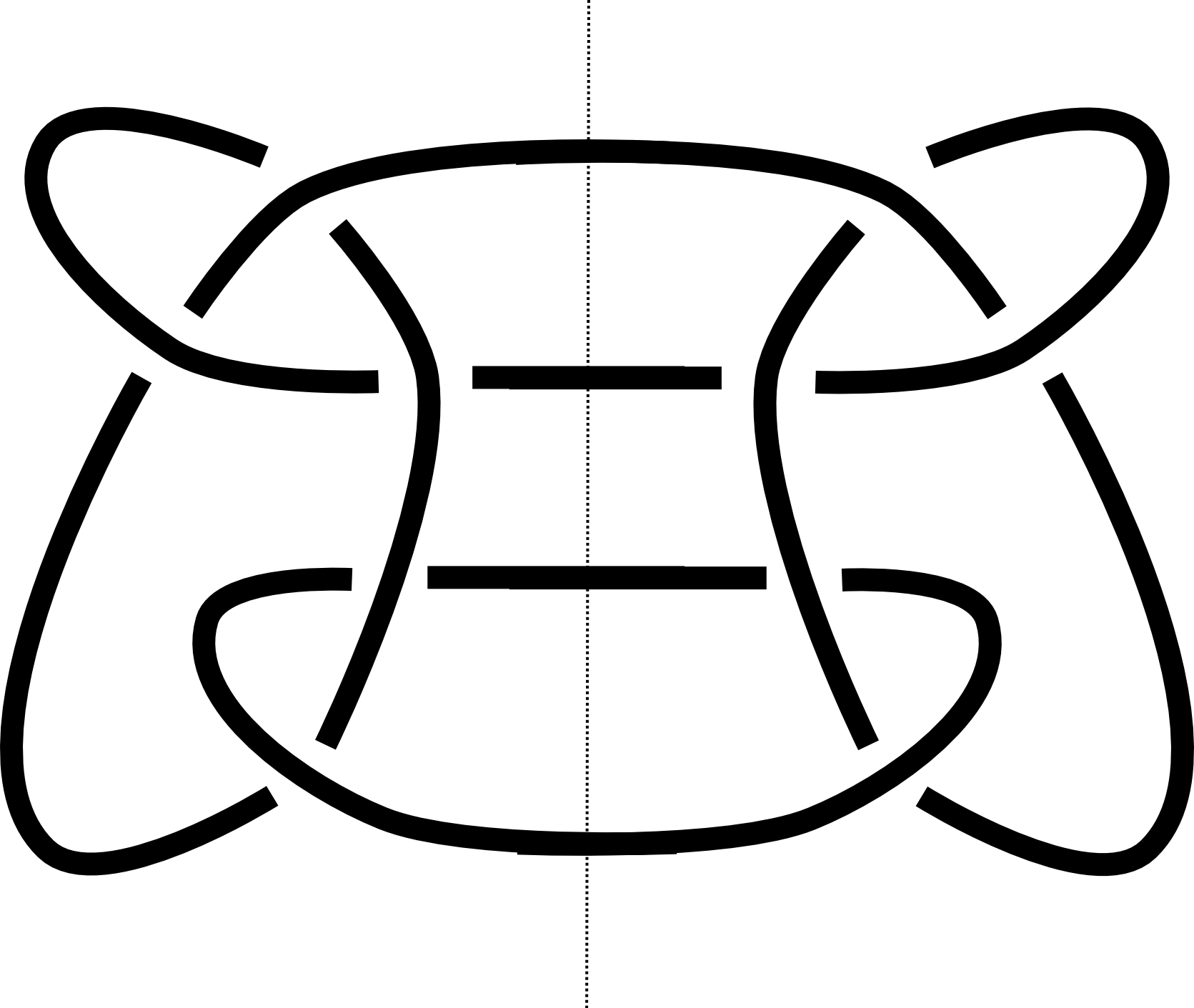} 
  \caption{A $L10^n_{36}$ link: orientations are irrelevant for \ref{prop:orientations}}
  \label{L10n36}
\end{figure}
this link is a symmetric union, a class of links introduced by Kinoshita and Terasaka in \cite{Kinoshita}, which are strongly 
slice. Moreover, 
we can say that it is non split because it is classified in Thistlethwaite's table, so this is the example we were looking for.

Since $s$ and $d$ are strong concordance invariants we have 
$$s(L10^n_{36})=s((L10^n_{36})^*)=s(\bigcirc\sqcup\bigcirc)=-1$$ and 
$$d_{L10^n_{36}}=d_{\bigcirc\sqcup\bigcirc}(i,s)=\left\{\begin{aligned}
                                                   &\binom{2}{k}\:\text{ if }i=0,s=2-2k \\
                                                   &0\:\:\:\:\:\:\:\:\text{ otherwise}\\
                                                  \end{aligned}\right.
                                                  \:\:\:\:\:\text{for }k=0,1,2\:.$$
So $s(L10^n_{36})\neq-s((L10^n_{36})^*)$ and this shows that Proposition \ref{prop:mirror_pseudo} does not hold in this case.
Indeed $L10^n_{36}$ is not pseudo-thin since its $d(0,\cdot)$ invariant is supported in 3 points. This means that the 
pseudo-thin hypothesis in \ref{prop:mirror_pseudo} is necessary.

$\newline$ Finally, we compute the value of the strong slice genus for some interesting links, using Theorem \ref{teo:mine}. 
We consider only one
relative orientation because, in all the following examples, links are algebrically split so $g^*_*$ does not change.

Let $Ti_n$, with $n\geq 0$, be links represented by the following diagram:
\begin{figure}[H]
  \centering
  \includegraphics[width=0.35\textwidth]{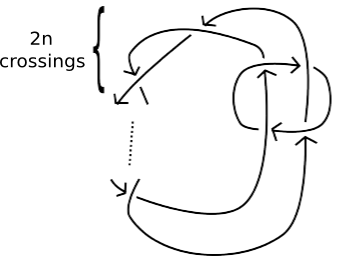}  
  \caption{A $Ti_n$ link}
\end{figure}
all these links are non split alternating, further we have $\text{lk}(L_1,L_2)=0$. Since they are alternating we can easily 
compute their $s$ invariant.

We show in Figure \ref{qwert} the coherent resolution of $Ti_n$ and, in this way, we obtain that $s(Ti_n)=2n+1$ and 
Theorem \ref{teo:mine} says that $$g_*^*(Ti_n)\geq\dfrac{2n+1+1}{2}=n+1\:.$$ \newpage
\begin{figure}[H]
  \centering
  \includegraphics[width=0.5\textwidth]{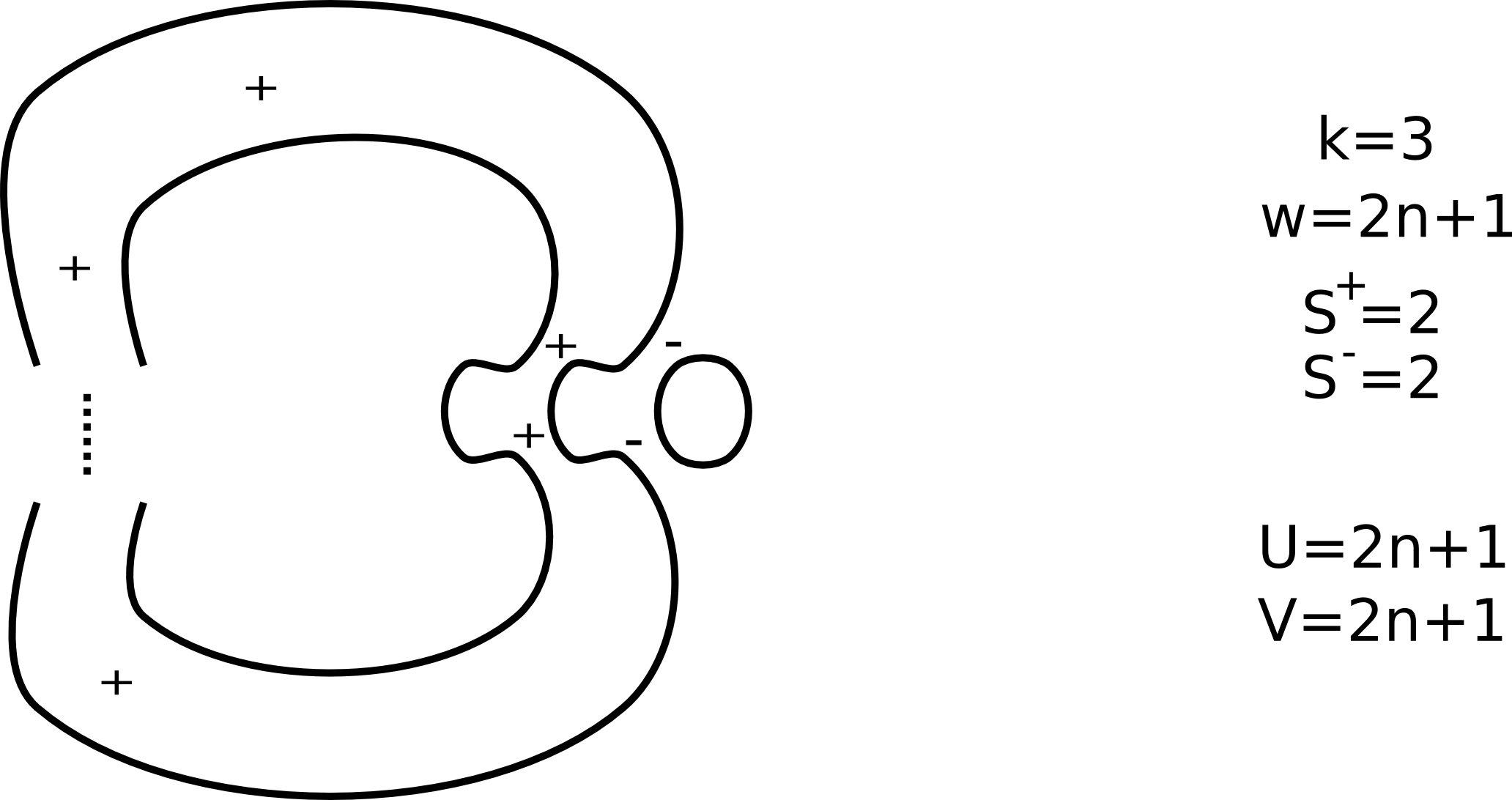}  
  \caption{}
  \label{qwert}
\end{figure} 
We apply a Reidemeister move and, in the highlighted point of next figure, a type 1 Morse move.
\begin{figure}[H]
  \centering
  \includegraphics[width=0.7\textwidth]{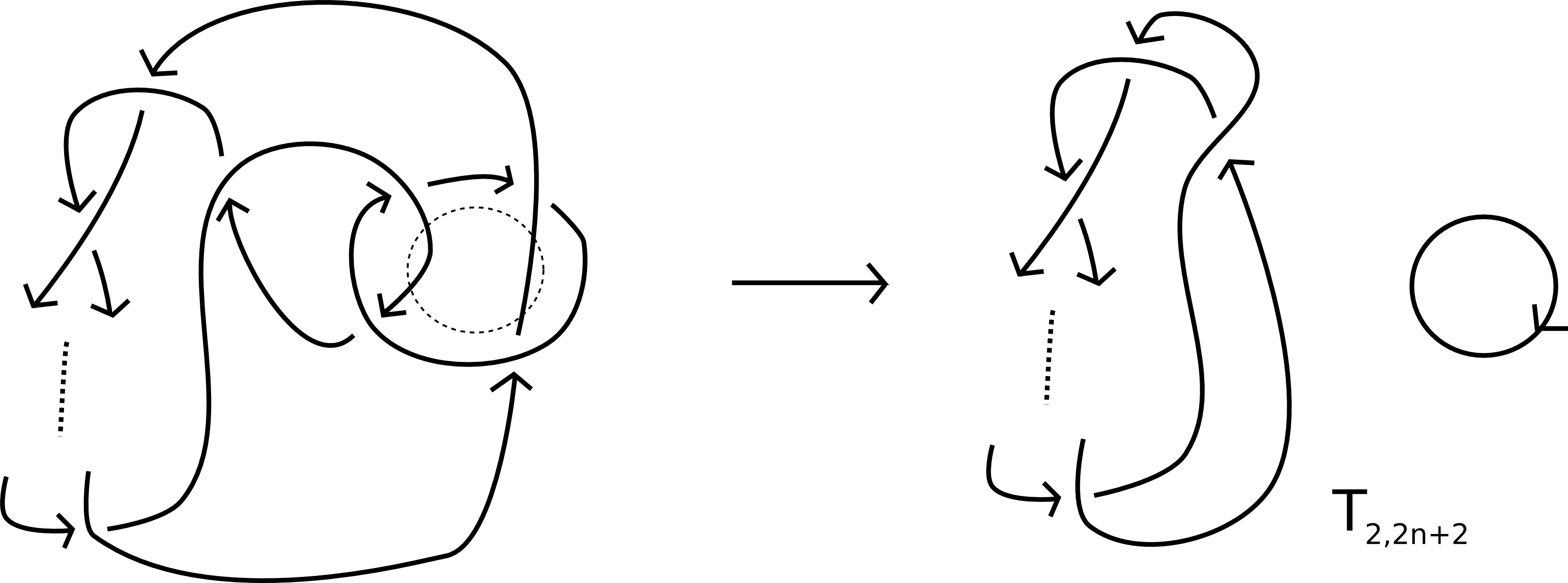}  
\end{figure}
We obtain that $Ti_n$ is cobordant to $\bigcirc\sqcup T_{2,2n+2}$ and we know the slice genus of this link: 
$g_*(T_{2,2n+2})=n$. Finally, we obtain the surface in the next figure.
\begin{figure}[H]
  \centering
  \includegraphics[width=0.45\textwidth]{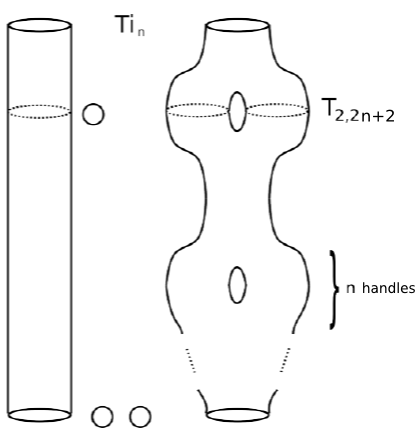}  
  \caption{}
\end{figure}
This is a strong cobordism of genus $n+1$ between $Ti_n$ and the unlink, so \\ $g_*^*(Ti_n)=n+1$.

For $n=0$ we have the Whitehead link.

If we take $n=1$ then $Ti_1=L7^a_3$. This link is an example where the slice genus and the strong slice genus are 
different: 
we have just shown that the second one is equal to 2, while, since $s(L7^a_3)=3$, we have $g_*(L7^a_3)\geq 1$ from Equation
\eqref{bound_inf_s}.
To see that the slice genus is exactly 1 we need to find a weak cobordism of genus 1 between $L7^a_3$ and the unknot. This 
construction is described in the following figures.
\begin{figure}[H]
  \centering
  \includegraphics[width=0.7\textwidth]{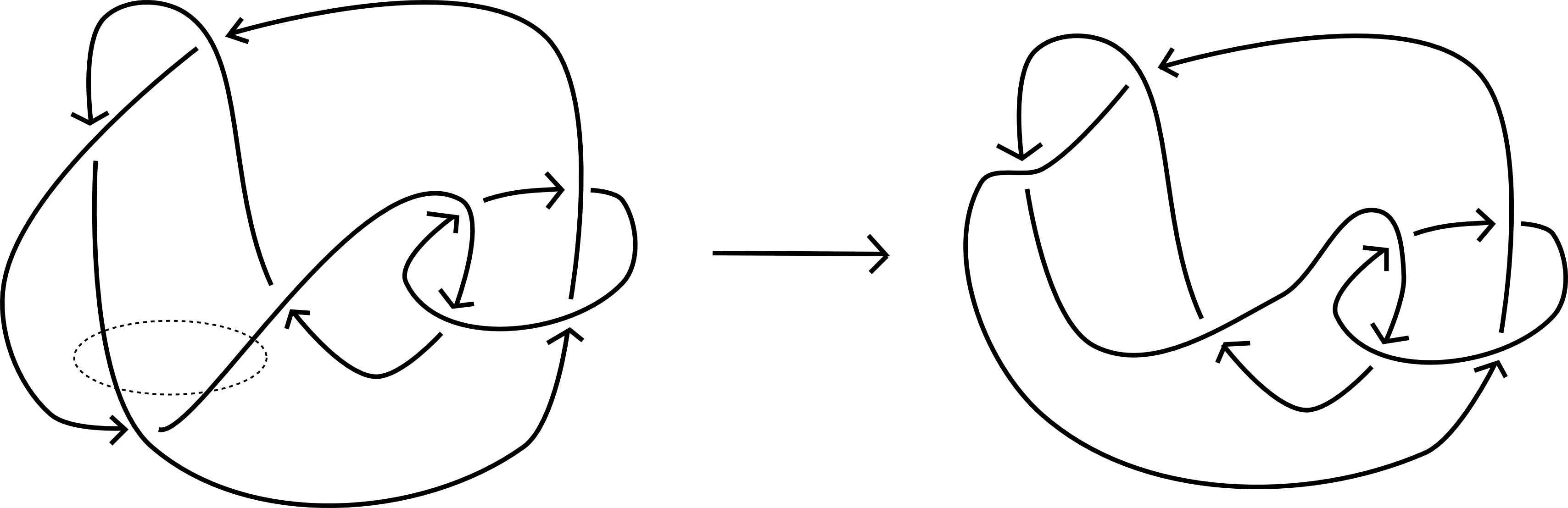}  
\end{figure} 
The figure on the left shows a $L7^a_3$ link, while the one on the right a $L6^a_5$ link. The highlighted tangles mark the points where we apply 
type 1 Morse moves. 
\begin{figure}[H]
  \centering
  \includegraphics[width=0.7\textwidth]{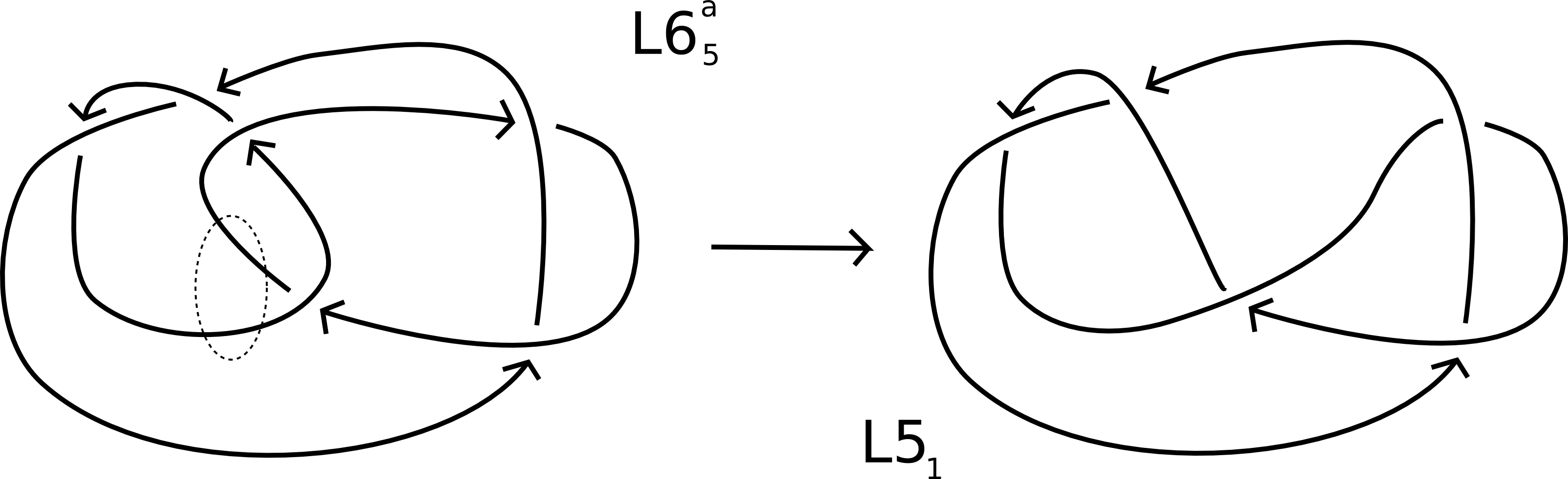}  
\end{figure}
Where $L5_1$ is the Whitehead link.
\begin{figure}[H]
  \centering
  \includegraphics[width=0.3\textwidth]{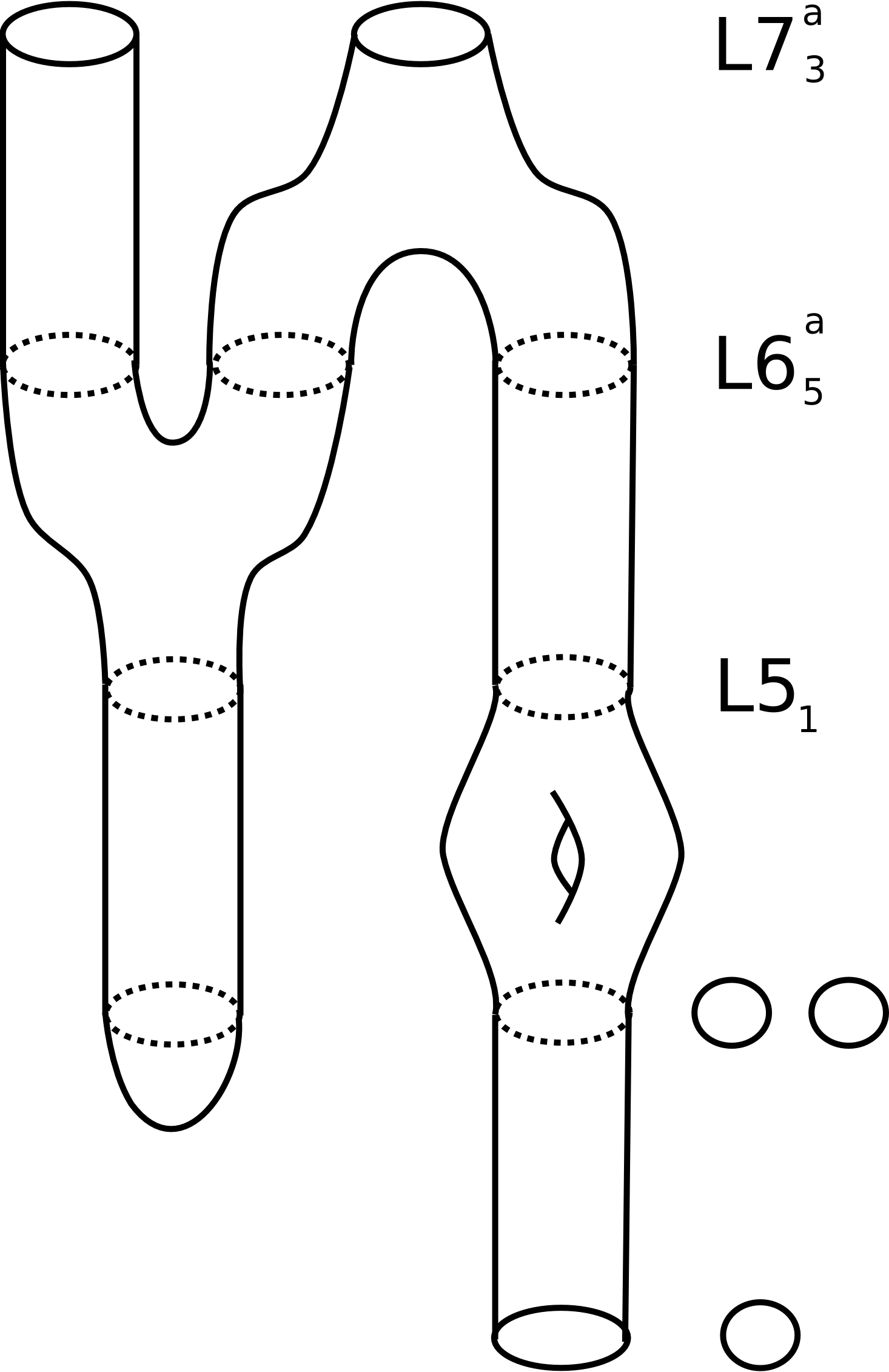}  
  \caption{}
\end{figure}  
The last cobordism has genus 1. In general, $g_*$ and $g_*^*$ are actually two distinct invariants.

$\newline$ Now we consider 3-strand pretzel links $P_{2h,2h,2l+1}$ with $h>0$. This is the only case we did not study in the 
previous section.
Such links are non split alternating and the linking number between their components is zero. Moreover, we have that 
$P_{2,2,2l+1}=Ti_l$ thus the previous example is simply a particular case of the following one.

We can compute the $s$ invariant as we did for the other pretzels, and we omit the details:
$$s(P_{2h,2h,2l+1})=V(D_{2h,2h,2l+1})=U(D_{2h,2h,2l+1})=2l+2h-1\:.$$ 
Theorem \ref{teo:mine} says $$g_*^*(P_{2h,2h,2l+1})\geq l+h\:.$$
We prove by induction that there exists a strong cobordism of genus $l+h$ between $P_{2h,2h,2l+1}$ and the unlink:

$\newline$ $h=1$
$\newline$ From previous result on $Ti_l=P_{2,2,2l+1}$ they have strong slice genus equal to $l+1$.

$\newline$ $h\rightarrow h+1$
$\newline$ We apply two type 1 Morse moves as shown in the following figure, which represents the upper part of the left and central 
strands of the standard diagram of $P_{2h+2,2h+2,2l+1}$.
\begin{figure}[H]
        \centering
        \includegraphics[width=0.78\textwidth]{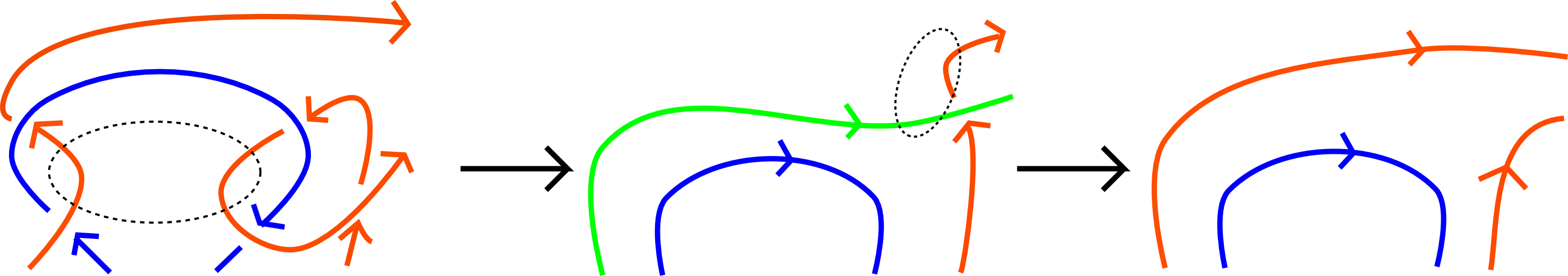}  
\end{figure}
The tangle on the right belongs to a pretzel link similar to $P_{2h+2,2h+2,2l+1}$, 
but with $h$ lowered by 1: we can use the inductive hypothesis.
       
We obtain the following cobordism.
\begin{figure}[H]
        \centering
        \includegraphics[width=0.45\textwidth]{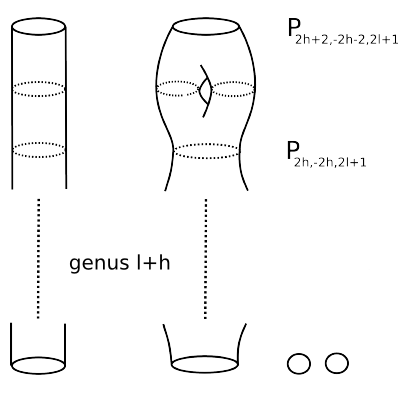}  
        \caption{}
\end{figure}
The genus of this surface is $l+h+1$ and so we conclude.

\bigskip


\begin{thebibliography}{9}
 \bibitem{Beliakova} A. Beliakova, S. Wehrli, \emph{Categorification of the colored Jones polynomial and Rasmussen invariant 
                                                    of links},  
              Canad. J. Math. \textbf{60(6)} (2008) 1240-1266. 
 \bibitem{Khovanov} M. Khovanov, \emph{A categorification of the Jones polynomial}, 
              Duke Math. J. \textbf{101(3)} (2000) 359-426.  
 \bibitem{Kinoshita} S. Kinoshita, H. Terasaka, \emph{On unions of knots}, 
              Osaka Math. J. \textbf{9} (1957) 131-153.             
 \bibitem{Lee} E. S. Lee, \emph{An endomorphism of the Khovanov invariant}, 
              Adv. Math. \textbf{197(2)} (2005) 554-586.
 \bibitem{Lobb} A. Lobb, \emph{Computable bounds for Rasmussen's concordance invariant}, 
              Compositio Mat. \textbf{147} (2011) 661-668.  
 \bibitem{Manolescu} C. Manolescu, P. Ozsv\'ath, \emph{On the Khovanov and knot Floer homologies of quasi-alternating links}, 
              In Proceedings of G\"{o}kova Geometry-Topology Conference (2007) 60-81.             
 \bibitem{Pardon} J. Pardon, \emph{The link concordance invariant from Lee homology}, 
              Algebr. Geom. Topol. \textbf{12(2)} (2012) 1081-1098.             
 \bibitem{Rasmussen} J. Rasmussen, \emph{Khovanov homology and the slice genus},
              Invent. Math. \textbf{182(2)} (2010) 419-447.    
 \bibitem{Turner} P. Turner, \emph{Five lectures on Khovanov homology}, 
              math.GT/0606464.             
\end{thebibliography}
\end{document}